\documentclass[11pt]{article}
\usepackage{declare-style-article}
\usepackage{declare-maths-article}
\usepackage{declare-text-article}
\usepackage{declare-theorems-article}
\hypersetup{
	pdfauthor 
		= {Nikita Nikolaev},
	pdftitle 
		= {Gevrey Asymptotic Implicit Function Theorem},
	pdfkeywords
		= {exact perturbation theory, singular perturbation theory, Borel summation, Borel-Laplace theory, asymptotic analysis, Gevrey asymptotics, resurgence, exact WKB analysis},
}

\DeclareSymbolFont{bbold}{U}{bbold}{m}{n}
\DeclareSymbolFontAlphabet{\mathbbold}{bbold}

\usepackage{appendix}
\usepackage[font={footnotesize}]{caption}
\usepackage[utf8]{inputenc}
\fancypagestyle{frontpage}{

\cfoot{}
\lfoot{\footnotesize 
First appeared: 16 December 2021\\
Contact: \href{mailto:n.nikolaev@sheffield.ac.uk}{n.nikolaev@sheffield.ac.uk}
}
}

\usepackage{titletoc}
\usepackage{subcaption}
\usepackage{changepage}

\usepackage[hyphenbreaks]{breakurl}

\titlecontents*{subsubsection}
  [5em]
  {\scriptsize\itshape}
  {}
  {}
  {}
  [\:\|\:]
  []
  
\newcommand{\ii}{{\bm{i}}}
\newcommand{\jj}{{\bm{j}}}

\renewcommand{\mm}{{\bm{m}}}
\renewcommand{\nn}{{\bm{n}}}

\newcommand{\MSCSubjectCode}[1]{\href{https://zbmath.org/classification/?q=cc\%3A#1}{#1}}
  
\begin{document}


\title{Gevrey Asymptotic Implicit Function Theorem}

\author{Nikita Nikolaev}

\affil{\small School of Mathematics and Statistics, University of Sheffield, United Kingdom}

\date{20 December 2021}

\maketitle
\thispagestyle{frontpage}

\begin{abstract}
\noindent
We prove an Asymptotic Implicit Function Theorem in the setting of Gevrey asymptotics with respect to a parameter.
The unique implicitly defined solution admits a Gevrey asymptotic expansion and furthermore it is the Borel resummation of the corresponding implicitly defined formal power series solution.
The main theorem can therefore be rephrased as an Implicit Function Theorem for Borel summable power series.
As an application, we give a diagonal or Jordan decomposition for holomorphic matrices in Gevrey asymptotic families.
\end{abstract}

{\small

\textbf{2020 MSC:}
	\MSCSubjectCode{58C15} (primary),
	\MSCSubjectCode{32A05},
	\MSCSubjectCode{40G10},
	\MSCSubjectCode{35C10}
}


{\begin{spacing}{0.9}
\small
\setcounter{tocdepth}{3}
\tableofcontents
\end{spacing}
}

\newpage
\section{Introduction}

Suppose $\FF (x, \hbar, z)$ is a holomorphic, perhaps vector-valued, function of several complex variables $x$ and $z$ and a small complex perturbation parameter $\hbar$ constrained to some sector at the origin in the $\hbar$-plane where $\FF$ admits an asymptotic expansion $\hat{\FF} (x, \hbar, z)$ as $\hbar \to 0$.
This paper arose from the following question: what is the meaning of a formal $\hbar$-power series solution $z = \hat{f} (x, \hbar)$ of the formal equation $\hat{\FF} (x, \hbar, z) = 0$?
The answer we find is that, provided sufficient control on the asymptotics of $\FF$, the formal solution $\hat{f}$ is the asymptotic expansion of an actual solution $z = f (x, \hbar)$ of the analytic equation $\FF (x, \hbar, z) = 0$, and furthermore $f$ is the Borel resummation of $\hat{f}$.
Thus, the purpose of this article is to prove the following version of the Implicit Function Theorem in the setting of Gevrey asymptotics.

\vspace{-5pt}

\begin{thm}[Gevrey Asymptotic Implicit Function Theorem]{211125100306}
\mbox{}\\
Fix a point $(x_0, z_0) \in \Complex^d_x \times \Complex^\NN_z$ with $d \geq 0$ and $\NN \geq 1$.
Let $X \subset \Complex^d_x$ be a domain containing $x_0$ and $S \subset \Complex_\hbar$ a sectorial domain with vertex at the origin and opening arc $A$ with opening angle $|A| = \pi$.
Suppose $\FF$ is a holomorphic map $X \times S \times \Complex^\NN_z \to \Complex^\NN$ which admits a Gevrey asymptotic expansion
\vspace{-10pt}
\eqntag{\label{211208193550}
	\FF (x, \hbar, z) 
		\simeq \hat{\FF} (x, \hbar, z)
		= \sum_{k=0}^\infty \FF_k (x,z) \hbar^k
\quad
\text{as $\hbar \to 0$ along $\bar{A}$
\fullstop{,}}\vspace{-10pt}
}
uniformly for all $x \in X$ and locally uniformly for all $z \in \Complex^\NN_z$.
Suppose its leading-order part in $\hbar$ satisfies $\FF_0 (x_0, z_0) = 0$ and the Jacobian $\del \FF_0 \big/ \del z$ is invertible at $(x_0, z_0)$.

Then there is a subdomain $X_0 \subset X$ containing $x_0$ and a sectorial subdomain $S_0 \subset S$ with the same opening $A$ such that there is a unique holomorphic map $f : X_0 \times S_0 \to \Complex^\NN$ which admits a Gevrey asymptotic expansion
\vspace{-10pt}
\eqntag{\label{211208194052}
	f (x, \hbar) \simeq \hat{f} (x, \hbar) = \sum_{n=0}^\infty f_n (x) \hbar^n
\quad
\text{as $\hbar \to 0$ along $\bar{A}$\fullstop{,}}\vspace{-10pt}
}
uniformly for all $x \in X_0$, and such that
\vspace{-10pt}
\eqntag{
	f_0 (x_0) = z_0
\qqtext{and}
	\FF \big(x, \hbar, f(x,\hbar) \big) = 0
\qquad
	\text{$\forall (x,\hbar) \in X_0 \times S_0$\fullstop}\vspace{-7pt}
}
Furthermore, $f$ is the uniform Borel resummation of $\hat{f}$ in the direction $\theta$ that bisects the arc $A$: for all $(x,\hbar) \in X_0 \times S_0$,
\vspace{-10pt}
\eqntag{\label{211125161647}
	f (x, \hbar) = \cal{S}_\theta \big[ \: \hat{f} \: \big] (x, \hbar)
\fullstop\vspace{-15pt}
}
\end{thm}
\enlargethispage{15pt}

This theorem provides a general answer in a large class of problems to the question of developing a theory of asymptotic implicit function theorems.
Such a question in a specialised setting was posed by Gérard and Jurkat in \cite[p.45]{MR1188077}, but to the best of our knowledge has not been addressed\footnote{In particular, the promised second part of their 1992 paper \cite{MR1188077} has not appeared.}.
In addition, our techniques give a much more refined information about the implicit function $f$, chiefly its uniform Borel summability properties.

\vspace{-5pt}

\paragraph{Application: Linear Algebra in Gevrey Asymptotic Families.}
As an application, which serves as the main source of motivation for us, \autoref{211125100306} can be used to diagonalise holomorphic matrices $\AA (x,\hbar)$ in uniform Gevrey asymptotic families; i.e., via transformations with the same regularity as $\AA$.
This means that the eigenvalues and the eigenspaces of $\AA$ are guaranteed to have the same asymptotic behaviour as $\hbar \to 0$ as the matrix $\AA$ itself.
More precisely, we prove the following diagonalisation theorem when the leading-order eigenvalues of $\AA$ are all distinct, which follows from the more general Jordan block decomposition \autoref{211210093539}.

\begin{thm}[{Diagonalisation in Gevrey Asymptotic Families}]{211209190924}
\mbox{}\\
Fix a domain $X \subset \Complex^d_x$ and a point $x_0 \in X$.
Let $S \subset \Complex_\hbar$ be a sectorial domain at the origin and opening arc $A$ with opening angle $|A| = \pi$.
Let $\AA = \AA (x, \hbar)$ be a holomorphic $n\!\!\times\!\!n$-matrix on $X \times S$ which admits a uniform Gevrey asymptotic expansion
\vspace{-5pt}
\eqntag{\label{211209195150}
	\AA (x, \hbar) \simeq \hat{\AA} (x, \hbar) = \sum_{k=0}^\infty \AA_{k} (x) \hbar^k
\quad
\text{as $\hbar \to 0$ along $\bar{A}$, unif. $\forall x \in X$\fullstop}\vspace{-5pt}
}
Suppose that the $\hbar$-leading-order part $\AA_{00} \coleq \AA_0 (x_0)$ at the point $x_0$ has distinct eigenvalues $a_1, \ldots, a_n \in \Complex$.
Let $\PP_{00}$ be an invertible $n\!\!\times\!\!n$-matrix that diagonalises $\AA_{00}$:
\vspace{-10pt}
\eqntag{\label{211210132517}
	\PP_{00}^\phantomindex \AA_{00}^\phantomindex \PP_{00}^{-1} = \diag \big( a_1, \ldots, a_n \big)
\fullstop
}
Then there is a subdomain $X_0 \subset X$ containing $x_0$ and a sectorial subdomain $S_0 \subset S$ with the same opening $A$ such that there is a unique holomorphic invertible $n\!\times\!n$-matrix $\PP = \PP (x, \hbar)$ on $X_0 \times S_0$ that admits a uniform Gevrey asymptotic expansion
\vspace{-5pt}
\eqntag{\label{211209195200}
	\PP (x, \hbar) \simeq \hat{\PP} (x, \hbar) = \sum_{k=0}^\infty \PP_{k} (x) \hbar^k
\quad
\text{as $\hbar \to 0$ along $\bar{A}$, unif. $\forall x \in X_0$\fullstop{,}}\vspace{-5pt}
}
such that $\PP_0 (x_0) = \PP_{00}$ and which diagonalises $\AA$; i.e.,
\vspace{-5pt}
\eqntag{\label{211210132609}
	\PP \AA \PP^{-1} = \diag \big( \lambda_1, \ldots, \lambda_n \big)
\fullstop\vspace{-5pt}
}
Furthermore, the transformation $\PP$ is the uniform Borel resummation of its asymptotic power series $\hat{\PP}$ in the direction $\theta$ that bisects the arc $A$: for all $(x,\hbar) \in X_0 \times S_0$,
\vspace{-5pt}
\eqntag{\label{211210134228}
	\PP = \cal{S}_\theta \big[ \, \hat{\PP} \, \big]
\fullstop\vspace{-10pt}
}
In addition:
\begin{enumerate}
\item The eigenvalues $\lambda_i = \lambda_i (x, \hbar)$ of $\AA$ are holomorphic functions on $X_0 \times S_0$ that admit uniform Gevrey asymptotic expansions
\vspace{-5pt}
\eqntag{\label{211209195205}
	\lambda_i (x, \hbar) \simeq \hat{\lambda}_i (x, \hbar) = \sum_{k=0}^\infty \lambda_{i,k} (x) \hbar^k
\quad
\text{as $\hbar \to 0$ along $\bar{A}$, unif. $\forall x \in X_0$\fullstop{,}}\vspace{-5pt}
}
with $\lambda_{i,0} (x_0) = a_i$.
Moreover, each eigenvalue $\lambda_i$ is the uniform Borel resummation of its asymptotic series $\hat{\lambda}_i$ in the direction $\theta$: for all $(x, \hbar) \in X_0 \times S_0$,
\vspace{-10pt}
\eqntag{\label{211210134826}
	\lambda_i = \cal{S}_\theta \big[ \, \hat{\lambda}_i \, \big]
\fullstop\vspace{-5pt}
}
\item Given an eigenbasis $v_1, \ldots, v_n \in \Complex^n$ for $\AA_{00}$, there is a unique eigenbasis $e_1, \ldots, e_n$ for $\AA$ consisting of holomorphic vectors $e_i = e_i (x, \hbar)$ on $X_0 \times S_0$ that admit uniform Gevrey asymptotic expansions
\vspace{-5pt}
\eqntag{\label{211210133236}
	e_i (x, \hbar) \simeq \hat{e}_i (x, \hbar)  = \sum_{k=0}^\infty e_{i,k} (x) \hbar^k
\quad
\text{as $\hbar \to 0$ along $\bar{A}$, unif. $\forall x \in X_0$\fullstop{,}}\vspace{-5pt}
}
with $e_{i,0} (x_0) = v_i$.
Moreover, each eigenvector $e_i$ is the uniform Borel resummation of its asymptotic series $\hat{e}_i$ in the direction $\theta$: for all $(x, \hbar) \in X_0 \times S_0$,
\vspace{-10pt}
\eqntag{\label{211210134828}
	e_i = \cal{S}_\theta \big[ \, \hat{e}_i \, \big]
\fullstop\vspace{-5pt}
}
\end{enumerate}
\end{thm}

Such results are useful in the exact WKB analysis of singularly perturbed meromorphic differential systems (and more generally singularly perturbed meromorphic connections on Riemann surfaces).
There, the role of the matrix $\AA (x, \hbar)$ is played by the principal part of the differential system at a pole.
For example, see \cite{nikolaev2019triangularisation} for the analysis of rank-two systems near a logarithmic pole.

\paragraph{Scalar polynomial case.}
Particularly notable for its simplicity and utility is the special case of \autoref{211125100306} where $\NN = 1$ and $\FF$ is a polynomial in the single variable $z$.
We restate it under these assumptions for ease of reference.

\begin{cor}{211217121003}
Fix a domain $X \subset \Complex_x^d$.
Let $S \subset \Complex_\hbar$ be a sectorial domain at the origin and opening arc $A$ with opening angle $|A| = \pi$.
Consider a polynomial
\eqntag{
	\FF = a_0 + a_1 z + \ldots + a_n z^n
}
whose coefficients $a_0, \ldots, a_m$ are holomorphic functions of $(x, \hbar) \in X \times S$ which admit uniform Gevrey asymptotic expansions
\eqntag{\label{211217121129}
	a_i (x, \hbar) 
		\simeq \hat{a}_i (x, \hbar)
		= \sum_{k=0}^\infty a_{i,k} (x) \hbar^k
\quad
\text{as $\hbar \to 0$ along $\bar{A}$, unif.$\forall x \in X$}
\fullstop}
Suppose that the leading-order discriminant
\eqn{
	\DD_0 
		= \DD_0 (x)
		\coleq \op{Disc}_x \big( \FF_0 \big)
		= \op{Disc}_x \big( a_{0,0} + a_{1,0} z + \ldots + a_{n,0} z^n \big)
}
is nonvanishing on $X$.
Let $z = f_0$ be a leading-order solution on $X$; i.e., a holomorphic function $f_0 (x)$ on $X$ such that $\FF_0 \big(x, f_0 (x)\big)$ for all $x \in X'$.
Then for any compactly contained subdomain $X_0 \subset X$, there is a sectorial subdomain $S_0 \subset S$ with the same opening $A$ such that the polynomial $\FF$ has a unique root $z = f(x, \hbar)$ which is a holomorphic function on $X_0 \times S_0$ and admits a uniform Gevrey asymptotic expansion \eqref{211208194052} with leading-order being the leading-order solution $f_0$.
Furthermore, $f$ is the uniform Borel resummation of $\hat{f}$ in the direction $\theta$ that bisects $A$.
\end{cor}

\paragraph{}
The asymptotic conditions \eqref{211208193550} and \eqref{211208194052} mean that both formal power series $\hat{\FF}$ and $\hat{f}$ are Borel-summable series in the direction $\theta$, so \autoref{211125100306} can be rephrased as an Implicit Function Theorem in the setting of Borel-summable series.

\enlargethispage{10pt}
\begin{cor}[Implicit Function Theorem for Borel-Summable Series]{211209140652}
\mbox{}\\
Fix a point $(x_0, z_0) \in \Complex^d_x \times \Complex^\NN_z$, a domain $X \subset \Complex^d_x$ containing $x_0$, and a direction $\theta$.
Let
\vspace{-5pt}
\eqntag{\label{211209140918}
	\hat{\FF} = \hat{\FF} (x,\hbar,z) = \sum_{k=0}^\infty \FF_k (x,z) \hbar^k
\vspace{-5pt}
}
be a formal power series in $\hbar$ whose coefficients $\FF_k$ are holomorphic maps $X \times \Complex_z^\NN \to \Complex^\NN$ such that $\FF_0 (x_0, z_0) = 0$ and the Jacobian $\del \FF_0 \big/ \del z$ is invertible at $(x_0, z_0)$.
Suppose $\hat{\FF}$ is Borel-summable in the direction $\theta$ uniformly for all $x \in X$ and locally uniformly for all $z \in \Complex_z^\NN$.
Then there is a subdomain $X_0 \subset X$ such that the unique formal series
\vspace{-5pt}
\eqntag{\label{211209141249}
	\hat{f} = \hat{f} (x,\hbar) = \sum_{n=0}^\infty f_n (x) \hbar^n
\fullstop{,}\vspace{-5pt}
}
with holomorphic coefficients $f_n : X_0 \to \Complex^\NN$, which satisfies $f_0 (x_0) = z_0$ as well as $\hat{\FF} \big(x,\hbar, \hat{f} (x,\hbar) \big) = 0$, is Borel-summable in the direction $\theta$ uniformly for all $x \in X_0$.
\end{cor}

After the initial release of this manuscript on the arXiv, we were alerted that a special case of this corollary (with $d = 0, \NN = 1$) was proved earlier by Kamimoto and Koike in \cite[Appendix A]{kamimoto2011borel} using a slightly different strategy.

\paragraph{Remarks and discussion.}
Our construction of the implicit function $f$ employs relatively basic and classical techniques from complex analysis which form the basis for the more modern and sophisticated theory of resurgent asymptotic analysis à la Écalle \cite{zbMATH03971144}; see also for instance \cite{MR2474083,sauzin2014introduction,MR3495546}.
Namely, we use the \textit{Borel-Laplace method}, also known as the theory of \textit{Borel-Laplace summability}.
We stress that the Borel-Laplace method ``is nothing other than the theory of Laplace transforms, written in slightly different variables'', echoing the words of Alan Sokal \cite{MR558468}.
As such, we have tried to keep our presentation very hands-on and self-contained, so the knowledge of basic complex analysis should be sufficient to follow.

We emphasise that the asymptotic condition \eqref{211208193550} on the holomorphic map $\FF$ is required to hold over the \textit{closed} arc $\bar{A} = [\theta - \tfrac{\pi}{2}, \theta + \tfrac{\pi}{2}]$, which is stronger than ordinary Gevrey asymptotics along an open arc $A$ (see \autoref{211215123326} or \cite[Appendix A.5 and A.16]{MY2008.06492} for a more detailed discussion).
This type of condition is exactly adapted to the Borel-Laplace method, see \autoref{211215123550}.
Similar methods are also used in the construction of exact WKB solutions for singularly perturbed ODEs such as the Schrödinger equation \cite{MY210623112236}.

What we call \textit{Gevrey asymptotics} is often called \textit{$1$-Gevrey asymptotics}.
It is part of an entire hierarchy of asymptotic regularity classes \cite{MR542737,MR579749}; see also \cite[\S1.2]{MR3495546}.
However, arguments about other Gevrey classes can usually be reduced to arguments about $1$-Gevrey asymptotics via a simple fractional transformation in the $\hbar$-space.
Therefore, we believe it is not difficult to extend our results to all other Gevrey asymptotic classes.
We leave this as a natural open problem.

\paragraph{Structure of the Paper.}
The proof of \autoref{211125100306} makes up all of \autoref{211125162211}.
Then in \autoref{211209190601}, we prove \autoref{211209190924} as well as its generalisation to Jordan blocks.
For pedagogical reasons, in \autoref{211125162208} we present the entire proof in the scalar case $\NN = 1$.
Although the contents of \autoref{211125162211} are strictly more general than those of \autoref{211125162208}, the two sections have been written in an entirely independent manner without any reference to each other.
We recommend the reader to begin with \autoref{211125162208} because it contains more or less all the essential ideas in the proof of the general vectorial case but without the added complication of having to keep track of many extra indices.

\enlargethispage{120pt}
\paragraph{Notation and conventions.}
\label{211214170249}
Our notation, conventions, and definitions from Gevrey asymptotics and Borel-Laplace theory are consistent with those given in Appendices A and B in \cite{MY2008.06492}.
A brief summary can be found in \autoref{211215112252}.

Throughout, we fix integers $d \geq 0$ and $\NN \geq 1$, and we write the vector components as $x = (x_1, \ldots, x_d)$, $z  = (z_1, \ldots, z_\NN)$, $\FF = (\FF^1, \ldots, \FF^\NN)$, $f = (f^1, \ldots, f^\NN)$.
The symbol $\Natural$ stands for nonnegative integers $0, 1, 2, \ldots$.
We use boldface letters to denote index vectors; i.e., $\mm \coleq (m_1, \ldots, m_\NN) \in \Natural^\NN$, etc., and we put $|\mm| \coleq m_1 + \cdots + m_\NN$.
Unless otherwise indicated, all sums over unbolded indices $n, m, \ldots$ are taken to run over $\Natural$, and all sums over boldface letters $\nn, \mm, \ldots$ are taken to run over $\Natural^\NN$.
Throughout this paper, we often suppress the explicit dependance on $x$ in the notation in the interest of brevity.

\paragraph*{Acknowledgements.}
The author wishes to thank Kohei Iwaki, Omar Kidwai, and Shinji Sasaki for helpful discussions.
The author also thanks Beatriz Navarro Lameda for her help in dealing with the many indices in some calculations.
This work was supported by the EPSRC Programme Grant \textit{Enhancing RNG}.

\section[Proof of \autoref*{211125100306}]{Proof of \autoref{211125100306}}
\label{211125162211}
\enlargethispage{5pt}

This section is dedicated to proving our main result, the Gevrey Asymptotic Implicit Function Theorem (\autoref{211125100306}).
The overall strategy of the proof is as follows.
First, we construct a formal solution $z = \hat{f}$ of the equation $\FF (x, \hbar, z) = 0$ using the ordinary Holomorphic Implicit Function Theorem at the leading-order in $\hbar$ and then using a recursion to determine all higher-order corrections.
We then want to apply the Borel resummation to $\hat{f}$ to get $f$.
To do so, we first make a convenient change of variables $z \mapsto w$ in order to put our equation into a certain standard form which is more amenable to the Borel transform.
Applying the Borel transform, we obtain a first-order ordinary differential equation for $\sigma = \Borel [ w ]$, albeit nonlinear and with convolution.
Nevertheless, this ODE is easy to convert into an integral equation, which we then proceed to solve using the method of successive approximations.
To show that this sequence of approximations converges to an actual solution $\sigma$, we give an estimate on the terms of this sequence by employing in an interesting way the ordinary Holomorphic Implicit Function Theorem.
This estimate also allows us to conclude that the Laplace transform $g = \Laplace [\sigma]$ of the obtained solution $\sigma$ exists and defines a holomorphic solution of our equation in standard form.
Undoing the change of variables $z \mapsto w$ sends $g$ to the desired solution $f$.

The proof is split into several intermediate lemmas.
All this work is finally put together on page \pageref{211215172900}.

\subsection{Formal Perturbation Theory}

The starting point is the following classical result whose proof is supplied below for completeness and in order to introduce some helpful notation.

\begin{prop}[{Formal Implicit Function Theorem}]{211209161918}
\mbox{}\\
Fix a domain $X \subset \Complex^d_x$ and a point $(x_0, z_0) \in X \times \Complex^\NN_z$.
Let
\vspace{-5pt}
\eqntag{\label{211127144949}
	\hat{\FF} = \hat{\FF} (x,\hbar,z) = \sum_{k=0}^\infty \FF_k (x,z) \hbar^k
\vspace{-5pt}
}
be a formal power series in $\hbar$ whose coefficients $\FF_k$ are holomorphic maps $X \times \Complex^\NN_z \to \Complex^\NN$ such that $\FF_0 (x_0, z_0) = 0$ and the Jacobian matrix $\del \FF_0 \big/ \del z$ is invertible at $(x_0, z_0)$.

Then there is subdomain $X_0 \subset X$ containing $x_0$ such that there is a unique formal power series
\vspace{-10pt}
\eqntag{\label{211206173102}
	\hat{f} = \hat{f} (x,\hbar) = \sum_{n=0}^\infty f_n (x) \hbar^n
}
whose coefficients $f_n$ are holomorphic maps $X_0 \to \Complex^\NN$, satisfying
\eqntag{
	f_0 (x_0) = z_0
\qqtext{and}
	\hat{\FF} \big( x, \hbar, \hat{f} (x,\hbar) \big) = 0
\qquad
	\text{$\forall x \in X_0$\fullstop}
}
In other words, the equation $\hat{\FF} (x, \hbar, z) = 0$ has a unique solution $z = \hat{f}$ defined near the point $x_0$ such that $f_0 (x_0) = z_0$.
In fact, all the higher-order coefficients $f_k$ are uniquely determined by $f_0$.

In particular, if $S \subset \Complex_\hbar$ is a sectorial domain at the origin, and $\FF$ is a holomorphic map $X \times S \times \Complex_z^\NN \to \Complex^\NN$ which admits the power series $\hat{\FF}$ as an asymptotic expansion as $\hbar \to 0$ in $S$, uniformly in $x$ and locally uniformly in $z$, then the equation $\FF (x,\hbar,z) = 0$ has a unique formal power series solution $z = \hat{f}$ near $x_0$ such that $f_0 (x_0) = z_0$.
\end{prop}

\begin{proof}
The proof amounts to plugging the solution ansatz \eqref{211206173102} into the formal equation $\hat{\FF} (x,\hbar,z) = 0$ and solving order-by-order in $\hbar$.
First, let us note down a few formulas in order to proceed with the calculation.
See also \autoref{211214170249}.

\textsc{Step 0: Collect some formulas.}
Write the double power series expansion of each component $\hat{\FF}^i$ as
\eqntag{\label{211208134110}
	\hat{\FF}^i (x, \hbar, z) 
		= \sum_{k=0}^\infty \sum_{m=0}^\infty \sum_{|\mm| = m}
			\FF^i_{k\mm} (x) \hbar^k z^\mm
\fullstop{,}
}
where $\FF^i_{k\mm} z^\mm \coleq \FF^i_{k m_1 \cdots m_\NN} z_1^{m_1} \cdots z_\NN^{m_\NN}$.
In particular, the expansion of the leading-order part $\FF_0$ is
\eqntag{\label{211208145851}
	\FF^i_0 (x, z) = \sum_{m=0}^\infty \sum_{|\mm| = m} \FF^i_{0\mm} (x) z^\mm
\fullstop
}
For every $\mm \in \Natural^\NN$, we have $\frac{\del}{\del z_j} z^\mm = \frac{m_j}{z_j} z^\mm$, so the $(i,j)$-component of the Jacobian matrix $\del \FF_0 \big/ \del z$ can be written as
\eqntag{\label{211208145846}
	\left[ \frac{\del \FF_0}{\del z} \right]_{ij}
		= \frac{\del \FF^i_0}{\del z_j}
		= \sum_{m=0}^\infty \sum_{|\mm| = m} \FF_{0\mm}^i (x) \frac{\del}{\del z_j} z^\mm
		= \sum_{m=0}^\infty \sum_{|\mm| = m} \frac{m_j}{z_j} \FF_{0\mm}^i (x) z^\mm
\fullstop
}
Next, the $\mm$-th power $\hat{f}^\mm$ of the power series ansatz \eqref{211206173102} expands as follows:
\eqns{
	\left( \sum_{n=0}^\infty f_n \hbar^n \right)^{\!\! \mm}
		&= 	\left( \sum_{n_1=0}^\infty f^1_{n_1} \hbar^{n_1} \right)^{\!\! m_1} \!\!\!\!\!
			\cdots
			\left( \sum_{n_\NN=0}^\infty f^\NN_{n_\NN} \hbar^{n_\NN} \right)^{\!\! m_\NN}
\\
		&=	\left( 
				\sum_{n_1=0}^\infty
				\sum_{|\jj_1| = n_1}^{\jj_1 \in \Natural^{m_1}}
					f_{j_{1,1}}^1 \cdots f_{j_{1,m_1}}^1 \hbar^{n_1}
			\right)
			\cdots
			\left( 
				\sum_{n_\NN=0}^\infty
				\sum_{|\jj_\NN| = n_\NN}^{\jj_\NN \in \Natural^{m_\NN}}
					f_{j_{\NN,1}}^\NN \cdots f_{j_{\NN,m_\NN}}^\NN \hbar^{n_\NN}
			\right)
\\
		&= \sum_{n=0}^\infty \sum_{|\nn| = n}
				\left(\sum_{|\jj_1| = n_1}^{\jj_1 \in \Natural^{m_1}}
					f_{j_{1,1}}^1 \cdots f_{j_{1,m_1}}^1
				\right)
				\cdots
				\left(\sum_{|\jj_\NN| = n_\NN}^{\jj_\NN \in \Natural^{m_\NN}}
					f_{j_{\NN,1}}^\NN \cdots f_{j_{\NN,m_\NN}}^\NN
				\right)
				\hbar^n
}
In these formulas, we have denoted the components of each vector $\jj_i \in \Natural^{m_i}$ by $(j_{i,1}, \ldots, j_{i,m_i})$.
Let us introduce the following shorthand notation:
\eqntag{\label{211208150122}
	\bm{f}^\mm_\nn 
		\coleq
		\left(\sum_{|\jj_1| = n_1}^{\jj_1 \in \Natural^{m_1}}
			f_{j_{1,1}}^1 \cdots f_{j_{1,m_1}}^1
		\right)
		\cdots
		\left(\sum_{|\jj_\NN| = n_\NN}^{\jj_\NN \in \Natural^{m_\NN}}
			f_{j_{\NN,1}}^\NN \cdots f_{j_{\NN,m_\NN}}^\NN
		\right)
\fullstop
}
We note the following simple but useful identities: 
\eqntag{
	\bm{f}^{\bm{0}}_{\bm{0}} = 1\fullstop{;}
	\qquad
	\bm{f}^{\mm}_{\bm{0}} = f_0^\mm = (f^1_0)^{m_1} \cdots (f^\NN_0)^{m_\NN}\fullstop{;}
	\qquad
	\bm{f}^{\bm{0}}_\nn = 0 \text{ whenever $|\nn| > 0$\fullstop}
}
Using this notation, the formula for $\hat{f}^\mm$ can be written much more compactly:
\eqntag{\label{211208150118}
	\hat{f}^\mm 
		= \left( \sum_{n=0}^\infty f_n \hbar^n \right)^{\!\! \mm}
		= \sum_{n=0}^\infty \sum_{|\nn| = n} \bm{f}^\mm_\nn \hbar^n
\fullstop
}

\newpage
\textsc{Step 1: Expand order-by-order.}
Now, we plug the solution ansatz \eqref{211206173102} into the equation $\hat{\FF} (x,\hbar,z) = 0$.
Using \eqref{211208134110} and \eqref{211208150118}, we find:
\eqntag{\label{211208150414}
	\sum_{n=0}^\infty \sum_{m=0}^\infty \sum_{k=0}^n
	\sum_{|\nn| = n - k} \sum_{|\mm| = m}
			\FF^i_{k\mm} \bm{f}^\mm_\nn \hbar^{n}
	= 0
\qqquad
\text{(\:$i = 1, \ldots, \NN$\:)\fullstop}
}
We solve \eqref{211208150414} for the coefficients $f_n$ order-by-order in $\hbar$.

\textsc{Step 2: Leading-order part.}
First, at order $n = 0$, equation \eqref{211208150414} yields:
\eqntag{\label{211208152123}
		\sum_{m=0}^\infty \sum_{|\mm| = m} \FF_{0\mm}^i (x) \bm{f}^\mm_{\bm{0}}
		= 0
\qqquad
\text{(\:$i = 1, \ldots, \NN$\:)\fullstop}
}
Comparing with \eqref{211208145851}, these equations are simply the components of the equation $\FF_0 (x, f_0) = 0$.
By the Holomorphic Implicit Function Theorem, there is a domain $X_0 \subset X$ containing $x_0$ such that there is a unique holomorphic map $f_0 : X_0 \to \Complex^\NN$ that satisfies $\FF_0 \big(x, f_0(x)\big) = 0$ and $f_0 (x_0) = z_0$.
In fact, the domain $X_0$ can be chosen so small that the Jacobian $\del \FF_0 \big/ \del z$ remains invertible at the point $\big(x, f_0(x)\big)$ for all $x \in X_0$.
Thus, we can define a holomorphic invertible $\NN\!\!\times\!\!\NN$-matrix $\JJ_0$ on $X_0$ by
\eqntag{\label{211208151359}
	\JJ_0 (x) \coleq \evat{\frac{\del \FF_0}{\del z}}{\big(x, f_0(x)\big)}
}
The $(i,j)$-component of $\JJ_0$ is:
\eqntag{\label{211208150755}
	[\JJ_0]_{ij}
		= \evat{\frac{\del \FF^i_0}{\del z_j}}{\big(x, f_0(x)\big)} \!\!\!\!
		= \sum_{m=0}^\infty \sum_{|\mm| = m} \frac{m_j}{f_0^j} \FF_{0\mm}^i \bm{f}^\mm_{\bm{0}}
\fullstop
}

\textsc{Step 3: Next-to-leading-order part.}
For clarity, let us also examine equation \eqref{211208150414} at order $n = 1$.
First, let us note that if $|\nn| = 1$, then $\nn = (0, \ldots, 1, \ldots, 0)$ with the only $1$ in some position $i$, in which case the notation \eqref{211208150122} reduces to:
\eqntag{\label{211208154854}
	\bm{f}^\mm_\nn
		= (f^1_0)^{m_1} \cdots \left(m_j f_1^j \right) (f_0^j)^{m_j-1} \cdots (f^\NN_0)^{m_\NN}
		= \frac{m_j}{f_0^j} \bm{f}_{\bm{0}}^\mm f_1^j
\fullstop
}
Then at order $n = 1$, equation \eqref{211208150414} comprises two main summands corresponding to $k = 0$ and $k = 1$, which simplify using identities \eqref{211208150755} and \eqref{211208154854}: 
\eqnstag{\nonumber
	\BLUE{\sum_{m=0}^\infty \sum_{|\mm| = m}
		\sum_{|\nn| = 1} \FF^i_{0\mm} \bm{f}^\mm_\nn}
	+ \sum_{m=0}^\infty \sum_{|\mm| = m}
		\FF^i_{1\mm} \bm{f}^\mm_{\bm{0}}
	&=0
\fullstop{,}
\\\nonumber
	\BLUE{\sum_{j=1}^\NN 
	\sum_{m=0}^\infty \sum_{|\mm| = m}
		\frac{m_j}{f_0^j} \FF^i_{0\mm} \bm{f}_{\bm{0}}^\mm f_1^j}
	+ \sum_{m=0}^\infty \sum_{|\mm| = m}
		\FF^i_{1\mm} \bm{f}^\mm_{\bm{0}}
	&=0
\fullstop{,}
\\\label{211208162829}
	\BLUE{\sum_{j=1}^\NN [\JJ_0]_{ij} f_1^j}
	+ \sum_{m=0}^\infty \sum_{|\mm| = m}
		\FF^i_{1\mm} \bm{f}^\mm_{\bm{0}}
	&=0
\fullstop
}
Observe that the \BLUE{blue} term is nothing but the $i$-th component of the vector $\JJ_0 f_1$.
Since $\JJ_0$ is an invertible matrix, multiplying the system of $\NN$ equations \eqref{211208162829} on the left by $\JJ^{-1}_0$, we solve uniquely for a holomorphic vector $f_1$ on $X_0$.

\newpage
\textsc{Step 4: Inductive step.}
Suppose now that $n \geq 1$ and we have already solved equation \eqref{211208150414} for holomorphic vectors $f_0, f_1, \ldots, f_{n-1}$ on $X_0$.
Similar to \eqref{211208154854}, we have that if $\nn = (0, \ldots, n, \ldots, 0)$ with the only nonzero entry in some position $j$, then
\eqntag{\label{211216093444}
	\bm{f}^\mm_\nn
		= (f^1_0)^{m_1} \cdots \left(m_j f_n^j \right) (f_0^j)^{m_j-1} \cdots (f^\NN_0)^{m_\NN}
		= \frac{m_j}{f_0^j} \bm{f}_{\bm{0}}^\mm f_n^j
\fullstop
}
Then at order $n$ in $\hbar$, we separate out the $k = 0$ summand and simplify using the identities \eqref{211208150755} and \eqref{211216093444}:
\vspace{-15pt}
\eqns{
	\sum_{m=0}^\infty \sum_{k=0}^n
	\sum_{|\nn| = n - k} \sum_{|\mm| = m}
			\FF^i_{k\mm} \bm{f}^\mm_\nn
	&=0
\fullstop{,}
\\
	\sum_{m=0}^\infty
	\left(
	\sum_{|\nn| = n} \sum_{|\mm| = m}
			\FF^i_{0\mm} \bm{f}^\mm_\nn
	+
	\sum_{k=1}^n \sum_{|\nn| = n - k} \sum_{|\mm| = m}
			\FF^i_{k\mm} \bm{f}^\mm_\nn
	\right)
	&=0
\fullstop{,}
\\
	\BLUE{\sum_{j=1}^\NN \sum_{m=0}^\infty
	\sum_{|\mm| = m}
			\frac{m_j}{f_0^j} \FF^i_{0\mm} \bm{f}_{\bm{0}}^\mm f_n^j}
		\hspace{0.55\textwidth}&
\\[-8pt]
	+ \sum_{m=0}^\infty 
	\left(
		\sum_{|\nn| = n}^{n_1, \ldots, n_\NN \neq n} \!\!
		\sum_{|\mm| = m}
				\FF^i_{0\mm} \bm{f}^\mm_\nn
		+
		\sum_{k=1}^n \sum_{|\nn| = n - k} \sum_{|\mm| = m}
				\FF^i_{k\mm} \bm{f}^\mm_\nn
	\right)
	&=0
\fullstop{,}
\\[5pt]
	\BLUE{\sum_{j=1}^\NN [\JJ_0]_{ij} f_n^j}
	+
	\sum_{m=0}^\infty 
	\left(
		\sum_{|\nn| = n}^{n_1, \ldots, n_\NN \neq n} \!\!
		\sum_{|\mm| = m}
				\FF^i_{0\mm} \bm{f}^\mm_\nn
		+
		\sum_{k=1}^n \sum_{|\nn| = n - k} \sum_{|\mm| = m}
				\FF^i_{k\mm} \bm{f}^\mm_\nn
	\right)
	&=0
\fullstop
}
The term in \BLUE{blue} is nothing but the $i$-th component of the vector $\JJ_0 f_n$.
Observe that the remaining part of this expression involves only the already-known components of the lower-order vectors $f_0, \ldots, f_{n-1}$.
Therefore, since $\JJ_0$ is invertible, multiplying this system of $\NN$ equations on the left by $\JJ_0^{-1}$, we can solve uniquely for the holomorphic vector $f_n$ on $X_0$. 
\end{proof}

\subsection{Transformation to the Standard Form}
\enlargethispage{10pt}

Next, we make a convenient change of variables in order to bring the given equation $\FF (x, \hbar, z) = 0$ to a standard form that is more easily handled using the Borel-Laplace method.
This transformation and the standard form are fully determined by the leading-order solution $f_0$ of the equation $\FF_0 (x,z) = 0$ and can always be achieved under our hypotheses.
Namely, we have the following statement.

\begin{lem}{211208172717}
Suppose $\FF$ is a holomorphic map $X \times S \times \Complex_z^\NN \to \Complex^\NN$ satisfying the hypotheses of \autoref{211209161918}.
Let $f_0$ and $f_1$ be the leading- and the next-to-leading-order parts of the formal solution $\hat{f}$ defined on $X_0 \subset X$.
Then the change of the unknown variable $z \mapsto w$ given by
\eqntag{\label{211208180449}
	z = f_0 + \hbar (f_1 + w)
}
transforms the equation $\FF (x, \hbar, z) = 0$ into an equation in $w$ of the form
\eqntag{\label{211208180500}
	w = \hbar \GG (x, \hbar, w)
\fullstop{,}
}
where $\GG$ is a holomorphic map $X_0 \times S \times \Complex_w^\NN \to \Complex^\NN$ uniquely determined by $f_0$ and $\FF$.
Furthermore, if $\FF$ admits a Gevrey asymptotic expansion as $\hbar \to 0$ along $\bar{A}$ uniformly for all $x \in X$ and locally uniformly for all $z \in \Complex^\NN_z$ and the domain $X_0$ is chosen so small that all the eigenvalues of $\JJ_0$ (where $\JJ_0$ is the invertible holomorphic matrix on $X_0$ given by \eqref{211208151359}) are bounded from below on $X_0$, then $\GG$ also admits a Gevrey asymptotic expansion as $\hbar \to 0$ along $\bar{A}$ uniformly for all $x \in X_0$ and locally uniformly for all $z \in \Complex^\NN_z$.
Specifically, $\GG$ is defined by
\eqntag{\label{211209144657}
	\GG (x, \hbar, w) \coleq \hbar^{-1} \Big( w - \hbar^{-1}\JJ_0^{-1} (x) \FF \big(x, \hbar, f_0 (x) + \hbar f_1 (x) + \hbar w\big)\Big)
\fullstop
}
\end{lem}

\enlargethispage{20pt}

\begin{proof}
The only thing to check is that the righthand side of \eqref{211209144657} has no negative powers in $\hbar$.
In particular, since each component of $\FF$ is an entire function in the variables $z_1, \ldots, z_\NN$, identity \eqref{211209144657} makes it obvious that $\GG$ admits a uniform Gevrey asymptotic expansion $\hat{\GG}$ as $\hbar \to 0$ along $\bar{A}$ whenever the eigenvalues of $\JJ_0$ are bounded from below and $\FF$ admits uniform Gevrey asymptotics.

Let us now verify that $\GG$ has no negative powers in $\hbar$.
Clearly, the leading-order part of $\FF \big(\hbar, f_0 + \hbar f_1 + \hbar w\big)$ is simply $\FF_0 \big(x, f_0 (x) \big)$ which is zero because $f_0$ is the leading-order solution.
Therefore, the righthand side of \eqref{211209144657} is at worst of order $\hbar^{-1}$.
We argue that the next-to-leading-order part of $\FF \big(\hbar, f_0 + \hbar f_1 + \hbar w\big)$ is equal to $\JJ_0 w$.
Evidently,
\eqntag{\label{211214180548}
	\Big[ \FF \big(\hbar, f_0 + \hbar f_1 + \hbar w\big) \Big]^{\OO (\hbar)}
	= \FF_1 (f_0) + \Big[ \FF_0 \big(f_0 + \hbar f_1 + \hbar w\big) \Big]^{\OO (\hbar)}
\fullstop
}
The $i$-th component of $\FF_1 (f_0)$ is easy to write down:
\eqntag{\label{211214180637}
	\FF_1^i (f_0) 
		= \sum_{m=0}^\infty \sum_{|\mm| = m}
			\FF^i_{1\mm} \bm{f}^\mm_{\bm{0}}
\fullstop
}
To expand the term $\big[ \FF_0 \big(f_0 + \hbar f_1 + \hbar w\big) \big]^{\OO (\hbar)}$, consider first the following calculation:
\eqns{
	&\phantom{=}~~
		\Big( f_0 + \hbar (f_1 + w) \Big)^\mm
\\	&= 
		\Big( f_0^1 + \hbar (f_1^1 + w_1) \Big)^{m_1}
		\cdots
		\Big( f_0^\NN + \hbar (f_1^\NN + w_\NN) \Big)^{m_\NN}
\\
	&=
		\left(
			\sum_{i_1 + j_1 = m_1} \!\!\!\!
			\binom{m_1}{i_1, j_1} 
			\big( f_0^1 \big)^{i_1} 
			\big(f_1^1 + w_1\big)^{j_1} \hbar^{j_1}
		\right)
		\cdots
		\left(
			\sum_{i_\NN + j_\NN = m_\NN} \!\!\!\!
			\binom{m_\NN}{i_\NN, j_\NN} 
			\big( f_0^\NN \big)^{i_\NN} 
			\big(f_1^\NN + w_\NN\big)^{j_\NN} \hbar^{j_\NN}
		\right)
\\
	&=
		\sum_{\substack{i_1 + j_1 = m_1 \\ \cdots \\ i_\NN + j_\NN = m_\NN}}^{\ii,\jj \in \Natural^\NN}
		\binom{m_1}{i_1, j_1} \cdots \binom{m_\NN}{i_\NN, j_\NN}
		f_0^\ii \big(f_1 + w\big)^{\jj}
		\hbar^{|\jj|}
\fullstop
}
We are only interested in the $|\jj|=1$ part of this sum.
This means $\jj = (0, \ldots, 1, \ldots, 0)$; i.e., for each $k = 1, \ldots, \NN$, we have $j_k = 1$, $i_k = m_k - 1$, and $j_{k'} = 0, i_{k'} = m_1$ for all $k' \neq k$.
Since $\binom{m_k}{m_k - 1, 1} = m_k$ and $\binom{m_{k'}}{m_{k'}, 0} = 1$, the coefficient of $\hbar$ in the above expression simplifies as follows:
\eqn{
	\sum_{k=1}^\NN 
		\frac{m_k}{f_0^k} \bm{f}^\mm_{\bm{0}}
		\big(f^k_1 + w_k\big)
\fullstop
}
Therefore, continuing \eqref{211214180548} and using the above calculation together with \eqref{211208150755} and \eqref{211214180637}, we find for every $j = 1, \ldots, \NN$:
\vspace{-5pt}
\eqns{
	\Big[ \FF^j \big(\hbar, f_0 + \hbar f_1 + \hbar w\big) \Big]^{\OO (\hbar)} \!\!\!
	&= \sum_{m=0}^\infty \sum_{|\mm| = m} \!\!
			\FF^j_{1\mm} \bm{f}^\mm_{\bm{0}}
		+ \sum_{k=1}^\NN \sum_{m=0}^\infty \sum_{|\mm| = m} \!\!
			\FF^j_{0\mm}
			\frac{m_k}{f_0^k} \bm{f}^\mm_{\bm{0}} \big(f^k_1 + w_k\big)
\displaybreak
\\
	&= \sum_{m=0}^\infty \sum_{|\mm| = m}
			\FF^j_{1\mm} \bm{f}^\mm_{\bm{0}}
		+ \sum_{k=1}^\NN [\JJ_0]_{kj} f_1^k 
		+ \sum_{k=1}^\NN [\JJ_0]_{kj} w_k
\fullstop
}
Using \eqref{211208162829}, it is now clear this this expression equals the $j$-th component of $\JJ_0 w$.
\end{proof}

The analogue of the Formal Implicit Function Theorem (\autoref{211209161918}) for equations of the form \eqref{211206173102} is especially easy to formulate.

\enlargethispage{20pt}

\begin{lem}{211209163535}
Let
\vspace{-5pt}
\eqntag{\label{211209163539}
	\hat{\GG} 
		= \hat{\GG} (x, \hbar, w) 
		\coleq \sum_{k=0}^\infty \GG_k (x, w) \hbar^k
\vspace{-5pt}
}
be any formal power series in $\hbar$ with holomorphic coefficients $\GG_k : X_0 \times \Complex_w^\NN \to \Complex^\NN$ for some domain $X_0 \subset \Complex^d_x$.
Then there is a unique formal power series
\vspace{-10pt}
\eqntag{\label{211209163813}
	\hat{g} = \hat{g} (x,\hbar) = \sum_{n=0}^\infty g_n (x) \hbar^n
\vspace{-10pt}
}
with holomorphic coefficients $g_k : X_0 \to \Complex^\NN$, which satisfies $\hat{g} (x, \hbar) = \hat{\GG} \big(x, \hbar, \hat{g} (x, \hbar)\big)$ for all $x \in X_0$.
In other words, the equation $w = \hbar \hat{\GG} (x, \hbar, w)$ has a unique formal power series solution $w = \hat{g} (x, \hbar)$.

In particular, if $S \subset \Complex_\hbar$ is a sectorial domain at the origin and $\GG$ is a holomorphic map $X_0 \times S \times \Complex^\NN_w \to \Complex^\NN$ which admits the power series $\hat{\GG}$ as a locally uniform asymptotic expansion as $\hbar \to 0$ in $S$, then the equation $w = \GG (x,\hbar,w) = 0$ has a unique formal power series solution $w = \hat{g} (x, \hbar)$ as above.

Moreover, $g_0 \equiv 0$ and all the higher-order coefficients $g_n$ are given by the following recursive formula: for every $i = 1, \ldots, \NN$,
\vspace{-5pt}
\eqntag{\label{211209163820}
	g_{n+1}^i
	= \sum_{k = 0}^{n} \sum_{m=0}^{n-k} \sum_{|\mm| = m} \sum_{|\nn| = n-k}
		\GG^i_{k\mm} \bm{g}^\mm_\nn
\fullstop{,}\vspace{-15pt}
}
where
\eqntag{\label{211209164708}
	\bm{g}^\mm_\nn 
		\coleq
		\left(\sum_{|\jj_1| = n_1}^{\jj_1 \in \Natural^{m_1}}
			g_{j_{1,1}}^1 \cdots g_{j_{1,m_1}}^1
		\right)
		\cdots
		\left(\sum_{|\jj_\NN| = n_\NN}^{\jj_\NN \in \Natural^{m_\NN}}
			g_{j_{\NN,1}}^\NN \cdots g_{j_{\NN,m_\NN}}^\NN
		\right)
\fullstop{,}
}
and where $\GG^i_{k\mm} = \GG^i_{k\mm} (x)$ are the coefficients of the double power series expansion
\vspace{-5pt}
\eqntag{\label{211209163824}
	\hat{\GG}^i (x, \hbar, w)
		= \sum_{k=0}^\infty \sum_{m=0}^\infty \sum_{|\mm| = m} \GG^i_{k\mm} (x) \hbar^k w^\mm
\fullstop\vspace{-10pt}
}
\end{lem}

\enlargethispage{10pt}

\begin{proof}
The proof is a computation very similar to the one in the proof of \autoref{211209161918}.
Plugging the solution ansatz \eqref{211209163813} into the double power series expansion \eqref{211209163824} of $\hat{\GG}^i$, the righthand side of the equation $w = \hbar \hat{\GG} (x, \hbar, w)$ becomes:
\vspace{-5pt}
\eqns{
	\hbar \sum_{k=0}^\infty \sum_{m=0}^\infty \sum_{|\mm| = m}
	 \GG^i_{k\mm} \hbar^k \left( \sum_{n=0}^\infty g_n \hbar^n \right)^\mm
	&= \hbar \sum_{k=0}^\infty \sum_{m=0}^\infty 
				\sum_{|\mm| = m} \sum_{n=0}^\infty \sum_{|\nn|=n}
			\GG^i_{k\mm} \bm{g}^\mm_\nn \hbar^{k+n}
\\	&= \hbar \ORANGE{\sum_{n=0}^\infty \sum_{k=0}^n} \sum_{m=0}^{\infty} 
				\sum_{|\mm| = m} \sum_{|\nn|=\ORANGE{n-k}}
			\GG^i_{k\mm} \bm{g}^\mm_\nn \hbar^{\ORANGE{n}}
\\	&= \hbar \sum_{n=0}^\infty \sum_{k=0}^n \sum_{m=0}^{\ORANGE{n-k}} 
				\sum_{|\mm| = m} \sum_{|\nn|=n-k}
			\GG^i_{k\mm} \bm{g}^\mm_\nn \hbar^{n}
\fullstop{,}\vspace{-10pt}
}
where in the last step we noticed that all terms with $m > |\nn| = n-k$ are zero because $g_0 \equiv 0$; cf. \eqref{211209164708}.
\end{proof}

\subsection{Gevrey Regularity of the Formal Solution}

Now we show that the formal Borel transform of the formal solution $\hat{f}$ is a convergent power series in the Borel variable $\xi$; that is, the coefficients $f_n$ grow not faster than $n!$.
More precisely, we prove the following proposition.

\begin{prop}[{Gevrey Formal Implicit Function Theorem}]{211209171812}
\mbox{}\\
Assume all the hypotheses of \autoref{211209161918} and suppose in addition that the power series $\hat{\FF}$ is locally uniformly Gevrey on $X \times \Complex_z^\NN$.
Then $X_0 \subset X$ can be chosen so small that the formal power series $\hat{f}$ is uniformly Gevrey on $X_0$.
In particular, the formal Borel transform 
\eqntag{
	\hat{\phi} (x, \xi) =
	\hat{\Borel} [ \, \hat{f} \, ] (x, \xi)
		\coleq \sum_{n=0}^\infty \tfrac{1}{n!} f_{n+1} (x) \xi^n
}
is a uniformly convergent power series in $\xi$.
Concretely, if $X_0 \subset X$ is any subset where all eigenvalues of $\JJ_0$ are bounded from below and such that there are $\AA, \BB > 0$ such that $|\FF_k (x, z)| \leq \AA \BB^k k!$ for all $k \geq 0$, uniformly for all $x \in X_0$ and for all $z \in \Complex_z^\NN$ with $|z| < \RR$ for some $\RR > 0$, then there are constants $\CC, \MM > 0$ such that
\eqntag{
	\big| f_k (x) \big| \leq \CC \MM^k k!
\qqquad
	\text{$\forall x \in X_0, \forall k$\fullstop}
}
\end{prop}

\begin{proof}
Let $X_0 \subset X$ be such that all the eigenvalues of the invertible holomorphic matrix $\JJ_0$ from \eqref{211208151359} are bounded from below.
Then, by \autoref{211208172717}, the proof boils down to proving the following claim.

\textbf{Claim.}
\textit{
Assume all the hypotheses of \autoref{211209163535} and suppose that the power series $\hat{\GG}$ is Gevrey uniformly for all $x \in X_0$ and locally uniformly for all $w \in \Complex_w^\NN$.
Then the formal solution $\hat{g}$ is also uniformly Gevrey on $X_0$.}

Let $\AA, \BB > 0$ be constants such that, for all $i = 1, \ldots, \NN$, all $k,m \in \Natural$, all $\mm \in \Natural^\NN$ such that $|\mm| = m$, and all $x \in X_0$,
\eqntag{\label{211209172430}
	\big| \GG^i_{k\mm} (x) \big| \leq \rho_m \AA \BB^{k+m} k!
\fullstop{,}
}
where $\rho_m$ is a normalisation constant defined as\footnote{The righthand side is the total number of weak compositions of $m$ into $\NN$ parts.}
\eqntag{\label{211209172628}
	\frac{1}{\rho_m} \coleq \sum_{|\mm| = m} 1 = \tbinom{m + \NN - 1}{\NN - 1}
\fullstop
}
We will show that there is a constant $\MM > 0$ such that
\eqntag{\label{211209172734}
	\big| g_{n+1}^i (x) \big| \leq \MM^{n+1} n!
\qqquad
	\text{$\forall x \in X_0, ~\forall n \in \Natural$\fullstop}
}
This bound will be demonstrated in two main steps.
First, we will recursively construct a sequence $\set{\MM_n}_{n=0}^\infty$ of nonnegative real numbers such that
\eqntag{\label{211209173300}
	\big| g_{n+1}^i (x) \big| \leq \MM_{n+1} n!
\qqquad
	\text{$\forall x \in X_0, ~\forall n \in \Natural$\fullstop}
}
Then we will show that there is a constant $\MM > 0$ such that $\MM_n \leq \MM^n$ for all $n$.

\newpage
\textsc{Step 1: Construction of $\set{\MM_n}_{n=0}^\infty$.}
Let $\MM_0 \coleq 0$.
We can take $\MM_1 \coleq \AA$ because $g_1^i = \GG^i_{0\bm{0}}$.
Now we use induction on $n$ and formula \eqref{211209163820}, which is more convenient to rewrite as follows:
\eqntag{\label{211209173912}
	g_{n+1}^i
	= \sum_{m=0}^\infty \sum_{k=0}^n \sum_{|\mm| = m} \sum_{|\nn| = n-k}
		\GG^i_{k\mm} \bm{g}^\mm_\nn
\fullstop
}
Notice that $\bm{g}^\mm_\nn = 0$ whenever $m = |\mm| > |\nn| = n-k$, so this expression really is the same as \eqref{211209163820}.
Assume that we have already constructed $\MM_0, \ldots, \MM_{n}$ such that $\big| g^i_{j} \big| \leq \MM_{j} (j-1)!$ for all $j = 1, \ldots, n$ and all $x \in X_0$.

Let us write down an estimate for $\bm{g}^\mm_\nn$ using formula \eqref{211209164708}:
\eqns{
	\big| \bm{g}^\mm_\nn \big|
	&\leq
		\left(\sum_{|\jj_1| = n_1}^{\jj_1 \in \Natural^{m_1}}
			\big| g_{j_{1,1}}^1 \big| \cdots \big| g_{j_{1,m_1}}^1 \big|
		\right)
		\cdots
		\left(\sum_{|\jj_\NN| = n_\NN}^{\jj_\NN \in \Natural^{m_\NN}}
			\big| g_{j_{\NN,1}}^\NN \big| \cdots \big| g_{j_{\NN,m_\NN}}^\NN \big|
		\right)
\\
	&\leq
		\left(\sum_{|\jj_1| = n_1}^{\jj_1 \in \Natural^{m_1}}
			\MM_{j_{1,1}} \cdots \MM_{j_{1,m_1}}
		\right)
		\cdots
		\left(\sum_{|\jj_\NN| = n_\NN}^{\jj_\NN \in \Natural^{m_\NN}}
			\MM_{j_{\NN,1}} \cdots \MM_{j_{\NN,m_\NN}}
		\right)
			\big(|\nn| - |\mm| \big)!
\fullstop{,}
}
where we repeatedly used the inequality $i!j!\leq(i+j)!$.
Introduce the following shorthand:
\eqntag{\label{211209180714}
	\bm{\MM}^\mm_\nn 
	\coleq 
	\left(\sum_{|\jj_1| = n_1}^{\jj_1 \in \Natural^{m_1}}
		\MM_{j_{1,1}} \cdots \MM_{j_{1,m_1}}
	\right)
	\cdots
	\left(\sum_{|\jj_\NN| = n_\NN}^{\jj_\NN \in \Natural^{m_\NN}}
		\MM_{j_{\NN,1}} \cdots \MM_{j_{\NN,m_\NN}}
	\right)
\fullstop
}
Then the estimate for $\bm{g}^\mm_\nn$ becomes simply $|\bm{g}^\mm_\nn| \leq \bm{\MM}^\mm_\nn \big( |\nn|-|\mm| \big)!$.
Now we can estimate $g^i_{n+1}$ using formula \eqref{211209173912}:
\eqns{
	|g^i_{n+1}|
	&\leq
		\sum_{k=0}^n
		\sum_{m=0}^\infty 
		\sum_{|\mm| = m} \sum_{|\nn| = n-k}
		\rho_m \AA \BB^{k+m} k! \bm{\MM}^\mm_\nn \big( n-k-m \big)!
\\
		&\leq \AA \sum_{k=0}^n
			\BB^k
		\sum_{m=0}^\infty
		\sum_{|\mm| = m} \sum_{|\nn| = n-k}
		\rho_m \BB^m \bm{\MM}^\mm_\nn n!
\fullstop
}
Thus, we can define
\eqntag{\label{211209181700}
	\MM_{n+1} 
		\coleq \AA
		\sum_{k=0}^n
			\BB^k
		\sum_{m=0}^\infty
		\sum_{|\mm| = m} \sum_{|\nn| = n-k}
		\rho_m \BB^m \bm{\MM}^\mm_\nn
\fullstop
}

\enlargethispage{10pt}
\textsc{Step 2: Construction of $\MM$.}
To see that $\MM_n \leq \MM^n$ for some $\MM > 0$, we argue as follows.
Consider the following pair of power series in an abstract variable $t$:
\eqntag{
	\hat{p} (t) \coleq \sum_{n=0}^\infty \MM_n t^n
\qtext{and}
	\QQ (t) \coleq \sum_{m=0}^\infty \BB^m t^m
\fullstop
}
Notice that $\hat{p} (0) = \MM_0 = 0$ and that $\QQ (t)$ is convergent.
We will show that $\hat{p} (t)$ is also convergent.
The key is the observation that they satisfy the following equation:
\eqntag{\label{211209182719}
	\hat{p} (t)
		= \AA t \QQ (t) \QQ \big( \hat{p} (t) \big)
		= \AA t \QQ (t) \sum_{m=0}^\infty \BB^m \hat{p}(t)^m
\fullstop
}
This equation was found by trial and error.
In order to verify it, we rewrite the power series $\QQ (t)$ in the following strange way:
\eqn{
	\QQ (t) = \sum_{m=0}^\infty \sum_{|\mm| = m} \rho_m \BB^m t^{\mm}
\fullstop{,}
}
where $t^\mm \coleq t^{m_1} \cdots t^{m_\NN} = t^m$.
Then \eqref{211209182719} is straightforward to verify directly by substituting the power series $\hat{p}(t)$ and $\QQ(t)$ and comparing the coefficients of $t^{n+1}$ using the defining formula \eqref{211209181700} for $\MM_{n+1}$.
Indeed, using the notation introduced in \eqref{211209180714}, we see that
\eqns{
	\hat{p} (t)^{\mm}
	&= \hat{p} (t)^{m_1} \cdots \hat{p} (t)^{m_\NN}
\\
	&= 	\left( \sum_{n_1=0}^\infty \MM_{n_1} t^{n_1} \right)^{\!\! m_1}
		\cdots
		\left( \sum_{n_\NN=0}^\infty \MM_{n_\NN} t^{n_\NN} \right)^{\!\! m_\NN}
\\
	&= 
	\left(\sum_{n_1=0}^\infty \sum_{|\ii_1| = n_1}^{\ii_1 \in \Natural^{m_1}}
		\MM_{i_{1,1}} \cdots \MM_{i_{1,m_1}}
		t^{n_1}
	\right)
	\cdots
	\left(\sum_{n_\NN=0}^\infty \sum_{|\ii_\NN| = n_\NN}^{\ii_\NN \in \Natural^{m_\NN}}
		\MM_{i_{\NN,1}} \cdots \MM_{i_{\NN,m_\NN}}
		t^{n_\NN}
	\right)
\\
	&=
	\sum_{n=0}^\infty
	\sum_{|\nn|=n}
	\left(\sum_{|\ii_1| = n_1}^{\ii_1 \in \Natural^{m_1}}
		\MM_{i_{1,1}} \cdots \MM_{i_{1,m_1}}
	\right)
	\cdots
	\left(\sum_{|\ii_\NN| = n_\NN}^{\ii_\NN \in \Natural^{m_\NN}}
		\MM_{i_{\NN,1}} \cdots \MM_{i_{\NN,m_\NN}}
	\right)
	t^n
\\
	&=
	\sum_{n=0}^\infty
	\sum_{|\nn|=n}
	\bm{\MM}^\mm_\nn t^n
\fullstop
}
Then the righthand side of \eqref{211209182719} expands as follows:
\eqns{
	&\phantom{=}~~
	\AA t \left( \sum_{k=0}^\infty \BB^k t^k \right)
		\left(
			\sum_{m=0}^\infty \sum_{|\mm| = m} \rho_m \BB^m 
			\big( \hat{p} (t) \big)^{\mm}
		\right)
\\
	&=
	\AA t \left( \sum_{k=0}^\infty \BB^k t^k \right)
		\left(
			\sum_{m=0}^\infty \sum_{|\mm| = m} \rho_m \BB^m
			\left(
				\sum_{n=0}^\infty
				\sum_{|\nn|=n}
				\bm{\MM}^\mm_\nn t^n
			\right)
		\right)
\\
	&=
	\AA t \left( \sum_{k=0}^\infty \BB^k t^k \right)
		\left( \sum_{n=0}^\infty \CC_n t^n \right)
\qtext{where}
	\CC_n \coleq \sum_{m=0}^\infty \sum_{|\mm| = m} \sum_{|\nn|=n} \rho_m \BB^m \bm{\MM}^\mm_\nn t^n
\\
	&=
	\AA t \sum_{n=0}^\infty \sum_{k=0}^n \BB^k \CC_{n-k} t^n
	= \sum_{n=0}^\infty 
		\left(
			\AA \sum_{k=0}^n \BB^k
			\sum_{m=0}^\infty \sum_{|\mm| = m} \sum_{|\nn|=n} \rho_m \BB^m \bm{\MM}^\mm_\nn
		\right)
		t^{n+1}
\fullstop{,}
}
which matches with \eqref{211209181700}.
Now, consider the following holomorphic function in two variables $(t,p)$:
\eqn{
	\FF (t,p) \coleq - p + \AA t \QQ(t) \QQ(p)
\fullstop
}
It has the following properties:
\enlargethispage{10pt}
\eqn{
	\FF (0,0) = 0
\qqtext{and}
	\evat{\frac{\del \FF}{\del p}}{(t,p) = (0,0)} = -1 \neq 0
\fullstop
}
By the Holomorphic Implicit Function Theorem, there exists a unique holomorphic function $p (t)$ near $t = 0$ such that $p (t) = 0$ and $\FF \big(t, p (t)\big) = 0$.
Thus, $\hat{p} (t)$ must be the convergent Taylor series expansion at $t = 0$ for $p(t)$, so its coefficients grow at most exponentially: i.e., there is a constant $\MM > 0$ such that $\MM_n \leq \MM^n$.
\end{proof}

\subsection{Exact Perturbation Theory}

Now we show that the convergent Borel transform $\hat{\phi} (x, \xi)$ of the formal solution admits an analytic continuation along a ray in the Borel $\xi$-plane and furthermore its Laplace transform is well-defined.
First, we prove the following lemma.

\begin{lem}{211124143506}
Let $X_0 \subset \Complex^d_x$ be a domain.
Let $S \coleq \set{\hbar ~\big|~ \Re (1/\hbar) > 1/\RR} \subset \Complex_\hbar$ be the Borel disc of some diameter $\RR > 0$.
Recall that its opening is $A_+ \coleq (-\pi/2, +\pi/2)$.
Let $\GG : X_0 \times S \times \Complex^\NN_w \to \Complex^\NN$ be a holomorphic map which admits a Gevrey asymptotic expansion
\eqntag{\label{211215140731}
	\GG (x, \hbar, w) \simeq \hat{\GG} (x, \hbar, w)
\quad
\text{as $\hbar \to 0$ along $\bar{A}_+$\fullstop{,}}
}
uniformly for all $x \in X_0$ and locally uniformly for all $w \in \Complex^\NN_w$.
Then there is a Borel disc $S_0 \coleq \set{\hbar ~\big|~ \Re (1/\hbar) > 1/\RR_0} \subset S$ of possibly smaller diameter $\RR_0 \in (0, \RR]$ such that there is a unique holomorphic map $g : X_0 \times S_0 \to \Complex^\NN$ which admits a uniform Gevrey asymptotic expansion
\eqntag{\label{211206174137}
	g (x, \hbar) \simeq \hat{g} (x, \hbar)
\quad
\text{as $\hbar \to 0$ along $\bar{A}_+$, unif. $\forall x \in X_0$\fullstop{,}}
}
and such that $g(x, \hbar) = \hbar \GG \big(x, \hbar, g(x,\hbar) \big) = 0$ for all $(x,\hbar) \in X_0 \times S_0$.
Furthermore, $g$ is the uniform Borel resummation of $\hat{g}$: for all $(x,\hbar) \in X_0 \times S_0$,
\eqntag{\label{211210124909}
	g (x, \hbar) = \cal{S} \big[ \: \hat{g} \: \big] (x, \hbar)
\fullstop
}
\end{lem}

\begin{proof}
First, uniqueness of $g$ follows from the asymptotic property \eqref{211206174137}.
Indeed, suppose $g'$ is another such map.
Then $g - g'$ is a holomorphic map ${X_0 \times S_0 \to \Complex^\NN}$ whose components are uniformly Gevrey asymptotic to $0$ as $\hbar \to 0$ along the closed arc $\bar{A}_+$ of opening angle $\pi$.
By Nevanlinna's Theorem (\cite[pp.44-45]{nevanlinna1918theorie} and \cite{MR558468}; see also \cite[Theorem B.11]{MY2008.06492}), there can only be one holomorphic function on $S_0$ (namely, the constant function $0$) which is Gevrey asymptotic to $0$ as $\hbar \to 0$ along $\bar{A}_+$.
Thus, each component of $g - g'$ must be identically zero.

To construct $g$, we start by expanding $\GG$ as a power series in $w$.
Each component $\GG^i$ of $\GG$ can be expressed as the following uniformly convergent multipower series in the components $w_1, \ldots, w_\NN$ of $w$:
\eqntag{\label{211124145949}
	\GG^i (x, \hbar, w) = \sum_{m=0}^\infty \sum_{|\mm| = m} \AA^i_\mm (x, \hbar) w^\mm
\qqquad
\text{(\:$i = 1, \ldots, \NN$\:)\fullstop{,}}
}
where $\AA^i_\mm w^\mm \coleq \AA^i_{m_1 \cdots m_\NN} w_1^{m_1} \cdots w_\NN^{m_\NN}$.
Then the vectorial equation $w = \hbar \GG (x, \hbar, w)$ can be written as the following coupled system of $\NN$ scalar equations:
\eqntag{\label{211124133406}
	w_i = \hbar \sum_{m=0}^\infty \sum_{|\mm| = m} \AA^i_\mm w^\mm
\qqquad
\text{(\:$i = 1, \ldots, \NN$\:)\fullstop}
}
It is convenient to separate the $m = 1$ term from the sum:
\eqntag{\label{211207142835}
	w_i = \hbar \AA^i_{\mathbf{0}} 
			+ \hbar \sum_{m=1}^\infty \sum_{|\mm| = m} \AA^i_\mm w^\mm
\qqquad
\text{(\:$i = 1, \ldots, \NN$\:)\fullstop}
}

\paragraph*{Step 1: The Borel Transform.}
Let $a_\mm^i = a_\mm^i (x)$ be the $\hbar$-leading-order part of $\AA_\mm^i$ and let $\alpha_\mm^i (x, \xi) \coleq \Borel \big[ \AA^i_\mm \big] (x, \xi)$.
By the assumption \eqref{211215140731}, there is some $\epsilon > 0$ such that each $\alpha_\mm^i$ is a holomorphic function on $X_0 \times \Xi$, where
\eqntag{
	\Xi \coleq \set{\xi ~\big|~ \op{dist} (\xi, \Real_+) < \epsilon}
\fullstop{,}
}
with uniformly at-most-exponential growth at infinity in $\xi$  (cf. \autoref{211215140949}), and
\eqntag{\label{211206123333}
	\AA^i_\mm (x, \hbar) = a^i_\mm (x) + \Laplace \big[\, \alpha^i_\mm \,\big] (x, \hbar)
}
for all $(x, \hbar) \in X_0 \times S$ provided that the diameter $\RR$ is sufficiently small.

Dividing each equation \eqref{211207142835} by $\hbar$ and applying the analytic Borel transform, we obtain the following system of $\NN$ coupled nonlinear ordinary differential equations with convolution:
\eqntag{\label{211206180828}
	\del_\xi \sigma^i
		= \alpha^i_{\mathbf{0}} 
		+ \sum_{m=1}^\infty \sum_{|\mm| = m} 
			\Big( a^i_\mm \sigma^{\ast \mm} + \alpha^i_\mm \ast \sigma^{\ast \mm} \Big)
\qqquad
\text{(\:$i = 1, \ldots, \NN$\:)\fullstop{,}}
}
where $\sigma^{\ast \mm} \coleq (\sigma^1)^{\ast m_1} \ast \cdots \ast (\sigma^\NN)^{\ast m_\NN}$ and the unknown variables $w_i$ and $\sigma^i$ are related by $\sigma^i = \Borel [w_i]$ and $w_i = \Laplace [\sigma^i]$.
A solution of the system \eqref{211206180828} with initial condition $\sigma (x, 0) = a_{\mathbf{0}} (x)$ is equivalently the solution of the following system of $\NN$ coupled integral equations:
\eqntag{\label{211124152233}
	\sigma^i = a^i_\mathbf{0}
		+ \int_0^\xi 
			\left[ \alpha^i_\mathbf{0} 
				+ \sum_{m=1}^\infty \sum_{|\mm| = m}
				\Big( a^i_\mm \sigma^{\ast \mm} + \alpha^i_\mm \ast \sigma^{\ast \mm} \Big)
			\right]
		\dd{t}
\quad
\text{(\:$i = 1, \ldots, \NN$\:)\fullstop{,}}
}
where the integral is taken along the straight line segment from $0$ to $\xi$.

\paragraph*{Step 2: Method of Successive Approximations.}
We solve this integral equation using the method of successive approximations.
To this end, define a sequence of holomorphic maps $\set{\sigma_n = (\sigma_n^1, \ldots, \sigma_n^\NN) : X_0 \times \Xi \to \Complex^\NN}_{n=0}^\infty$, as follows: for each $i = 1, \ldots, \NN$, let
\vspace{-10pt}\enlargethispage{10pt}
\eqntag{\label{211206181512}
	\sigma_0^i \coleq a^i_{\mathbf{0}},
\qquad
	\sigma_1^i 
		\coleq \int_0^\xi 
		\left[ \alpha^j_\mathbf{0} 
			+ \sum_{|\mm| = 1} a^i_\mm \sigma_0^\mm
		\right] \dd{t}
\fullstop{,}\vspace{-10pt}
}
and for all $n \geq 2$,
\vspace{-10pt}
\eqntag{\label{211124133914}
	\sigma^i_n
		\coleq \int_0^\xi
			\sum_{m=1}^n
			\sum_{|\mm|=m}
			\left[
				a_\mm^i \sum_{|\nn|=n-m} \bm{\sigma}_\nn^\mm
				+ \alpha_\mm^i \ast \sum_{|\nn|=n-m-1} \bm{\sigma}_\nn^\mm
			\right] \dd{t}
\fullstop
}
Here, for any $\nn, \mm \in \Natural^\NN$, we have introduced the notation
\eqntag{\label{211207111746}
	\bm{\sigma}_\nn^\mm
	\coleq
		\left(
			\sum_{|\jj_1| = n_1}^{\jj_1 \in \Natural^{m_1}}
			\sigma^1_{j_{1,1}} \ast \cdots \ast \sigma^1_{j_{1,m_1}}
		\right)
			\ast
			\cdots
			\ast
		\left(
			\sum_{|\jj_\NN| = n_\NN}^{\jj_\NN \in \Natural^{m_\NN}}
			\sigma^\NN_{j_{\NN,1}} \ast \cdots \ast \sigma^\NN_{j_{\NN,m_\NN}}
		\right)
\fullstop
}
Let us also note the following simple but useful identities: 
\begin{gather}
	\bm{\sigma}^{\bm{0}}_{\bm{0}} = 1\fullstop{;}
	\qquad
	\bm{\sigma}^{\bm{0}}_\nn = 0 \text{ whenever $|\nn| > 0$\fullstop{;}}
\\
\label{211215153152}
	\bm{\sigma}^{\mm}_{\bm{0}}
		= (\sigma^1_0)^{\ast m_1} \ast \cdots \ast (\sigma^\NN_0)^{\ast m_\NN}
		= \tfrac{1}{(m-1)!}\sigma_0^\mm \xi^{m-1}
\fullstop
\end{gather}

\enlargethispage{10pt}
\textbf{Main Technical Claim.}
\textit{The infinite series
\vspace{-5pt}
\eqntag{\label{211206182147}
	\sigma (x, \xi) \coleq \sum_{n=0}^\infty \sigma_n (x, \xi)
\vspace{-5pt}
}
converges uniformly for all $(x, \xi) \in X_0 \times \Xi$ and defines a holomorphic solution of the integral equation \eqref{211124152233} with uniformly at-most-exponential growth at infinity in $\xi$; that is, there are constants $\DD, \KK > 0$ such that, for each $i = 1, \ldots, \NN$,
\eqntag{\label{211214194634}
	\big| \sigma^i (x, \xi) \big| \leq \DD e^{\KK |\xi|}
\qquad
\text{$\forall (x, \xi) \in X_0 \times \Xi$\fullstop}
}
Furthermore, the convergent formal Borel transform
\vspace{-5pt}
\eqntag{
	\hat{\sigma} (x, \xi) =
	\hat{\Borel} [ \, \hat{g} \, ] (x, \xi)
		= \sum_{n=0}^\infty \tfrac{1}{n!} g_{n+1} (x) \xi^n
\vspace{-5pt}
}
of the unique formal solution $\hat{g}$ is the Taylor series expansion of $\sigma$ at $\xi = 0$.
}

The assertions of \autoref{211124143506} follow from this claim by defining 
\eqntag{\label{211214194729}
	g (x, \hbar) \coleq \Laplace [\sigma] (x, \hbar) = \int_0^{+\infty} e^{-\xi/\hbar} \sigma (x, \xi) \dd{\xi}
\fullstop
}
Indeed, the exponential estimate \eqref{211214194634} implies that the Laplace transform \eqref{211214194729} is uniformly convergent for all $(x, \hbar) \in X_0 \times S_0$ where $S_0 = \set{\hbar ~\big|~ \Re (1/\hbar) > 1/\RR_0}$ as long as $\RR_0 < \KK^{-1}$.
We now turn to the proof of the Main Technical Claim.

\paragraph*{Step 3: Solution Check.}
\enlargethispage{20pt}
First, assuming that the infinite series $\sigma$ is uniformly convergent for all $(x, \xi) \in X_0 \times \Xi$, we verify that it satisfies the integral equation \eqref{211124152233} by direct substitution.
Thus, the righthand side of \eqref{211124152233} becomes:
\eqntag{\label{211206194021}
	a^i_\mathbf{0} + \int_0^\xi 
		\left[ \alpha^i_\mathbf{0}
			+ \BLUE{\sum_{m=1}^\infty \sum_{|\mm|=m} 
				a^i_\mm \left(\:\sum_{n=0}^\infty \sigma_n\right)^{\!\!\!\ast \mm}}
			+ \sum_{m=1}^\infty \sum_{|\mm|=m}
				\alpha^i_\mm \ast \left(\:\sum_{n=0}^\infty \sigma_n\right)^{\!\!\!\ast \mm}
		\right] \dd{t}
\GREY{.}
}
Using the notation introduced in \eqref{211207111746}, the $\mm$-fold convolution product of the infinite series $\sigma$ expands as follows:
\eqns{
	&\phantom{=}~~
	\left(\:\sum_{n=0}^\infty \sigma_n\right)^{\!\!\!\ast \mm}
\\		&= 	\left(\:
				\sum_{n_1=0}^\infty \sigma_{n_1}^1
			\right)^{\!\!\!\ast m_1} \!\!\!\!\!\!\!
			\ast \cdots \ast
			\left(\:
				\sum_{n_\NN=0}^\infty \sigma_{n_\NN}^\NN
			\right)^{\!\!\!\ast m_\NN}
\\
		&= 	\left(\:
				\sum_{n_1=0}^\infty \sum_{|\jj_1| = n_1}^{\jj_1 \in \Natural^{m_1}}
				\sigma^1_{j_{1,1}} \ast \cdots \ast \sigma^1_{j_{1,m_1}}
			\right)
			\ast \cdots \ast
			\left(\:
				\sum_{n_\NN=0}^\infty \sum_{|\jj_\NN| = n_\NN}^{\jj_\NN \in \Natural^{m_\NN}}
				\sigma^\NN_{j_{\NN,1}} \ast \cdots \ast \sigma^\NN_{j_{\NN,m_\NN}}
			\right)
\\
		&=	\sum_{n=0}^\infty \sum_{|\nn|=n}
			\left(
				\sum_{|\jj_1| = n_1}^{\jj_1 \in \Natural^{m_1}}
				\sigma^1_{j_{1,1}} \ast \cdots \ast \sigma^1_{j_{1,m_1}}
			\right)
				\ast
				\cdots
				\ast
			\left(
				\sum_{|\jj_\NN| = n_\NN}^{\jj_\NN \in \Natural^{m_\NN}}
				\sigma^\NN_{j_{\NN,1}} \ast \cdots \ast \sigma^\NN_{j_{\NN,m_\NN}}
			\right)
\\
		&= 	\sum_{n=0}^\infty \sum_{|\nn|=n}
			\bm{\sigma}_\nn^\mm
\fullstop
}
Use this to rewrite the \BLUE{blue} terms in \eqref{211206194021}, separating out first the \GREEN{$m=1$ part} and then the \ORANGE{$(m,n)=(1,1)$ part} using the identity \eqref{211215153152}:
\eqns{
	&\phantom{=}~~~
	\BLUE{\sum_{m=1}^\infty \sum_{|\mm|=m} 
				a^i_\mm \left(\:\sum_{n=0}^\infty \sigma_n\right)^{\!\!\!\ast \mm}}
	\!\!\!\!
\\
	&=
		\GREEN{\sum_{|\mm|=1} 
				a^i_\mm 
				\sum_{n=0}^\infty \sum_{|\nn|=n}
				\bm{\sigma}_\nn^\mm
				}
		+ \sum_{m=2}^\infty \sum_{|\mm|=m} 
				a^i_\mm
				\sum_{n=0}^\infty \sum_{|\nn|=n}
				\bm{\sigma}_\nn^\mm
\\
	&=
		\ORANGE{\sum_{|\mm|=1} 
				a^i_\mm \sigma_0^\mm
				}
				+
				\sum_{|\mm|=1} 
				a^i_\mm 
				\sum_{n=1}^\infty \sum_{|\nn|=n}
				\bm{\sigma}_\nn^\mm
		+ \sum_{m=2}^\infty \sum_{|\mm|=m} 
				a^i_\mm
				\sum_{n=0}^\infty \sum_{|\nn|=n}
				\bm{\sigma}_\nn^\mm
\fullstop
}
Substituting this back into \eqref{211206194021} and using \eqref{211206181512}, we find:
\begin{multline}
\label{211207125028}
	\sigma^i_0 + \ORANGE{\sigma^i_1} + \int_0^\xi
		\left[
			\sum_{|\mm|=1} 
				a^i_\mm 
				\sum_{n=1}^\infty \sum_{|\nn|=n}
				\bm{\sigma}_\nn^\mm
			+ \sum_{m=2}^\infty \sum_{|\mm|=m} 
				a^i_\mm
				\sum_{n=0}^\infty \sum_{|\nn|=n}
				\bm{\sigma}_\nn^\mm
		\right.
\\
		\left.
			+ \sum_{m=1}^\infty \sum_{|\mm|=m}
				\alpha^i_\mm \ast 
				\sum_{n=0}^\infty \sum_{|\nn|=n}
				\bm{\sigma}_\nn^\mm
		\right] \dd{t}
\fullstop
\end{multline}
The goal is to show that the the integral in \eqref{211207125028} is equal to $\sum_{n \geq 2} \sigma^i_n$.
Focus on the expression inside the integral:
\eqn{
	{\sum_{|\mm|=1} \!\!
		a^i_\mm 
		\sum_{n=1}^\infty \sum_{|\nn|=n} \!\!
		\bm{\sigma}_\nn^\mm}
	+ \BLUE{\sum_{m=2}^\infty \sum_{|\mm|=m} \!\!
		a^i_\mm
		\sum_{n=0}^\infty \sum_{|\nn|=n} \!\!
		\bm{\sigma}_\nn^\mm}
	+ \GREEN{\sum_{m=1}^\infty \sum_{|\mm|=m} \!\!\!
		\alpha^i_\mm \ast \!
		\sum_{n=0}^\infty \sum_{|\nn|=n} \!\!
		\bm{\sigma}_\nn^\mm}
\fullstop
}
Shift the summation index $n$ up by $1$ in the black sum, by $m$ in the \BLUE{blue} sum, and by $m+1$ in the \GREEN{green} sum:
\eqn{
	{\sum_{|\mm|=1} \!\!
		a^i_\mm 
		\sum_{n=\ORANGE{2}}^\infty \sum_{|\nn|=\ORANGE{n-1}} \!\!\!\!
		\bm{\sigma}_\nn^\mm}
	+ \BLUE{\sum_{m=2}^\infty \sum_{|\mm|=m} \!\!\!
		a^i_\mm
		\sum_{n=\ORANGE{m}}^\infty \sum_{|\nn|=\ORANGE{n-m}} \!\!\!\!\!
		\bm{\sigma}_\nn^\mm}
	+ \GREEN{\sum_{m=1}^\infty \sum_{|\mm|=m} \!\!\!
		\alpha^i_\mm \ast \!\!\!\!
		\sum_{n=\ORANGE{m+1}}^\infty \sum_{|\nn|=\ORANGE{n-m-1}} \!\!\!\!\!\!
		\bm{\sigma}_\nn^\mm}
\fullstop
}
Notice that all terms in the \BLUE{blue} sum with $n < m$ are zero, so we can start the summation over $n$ from $n = 2$ (which is the lowest possible value of $m$) without altering the result.
Similarly, all terms in the \GREEN{green} sum with $n < m + 1$ are zero, so we may as well start from $n = 2$.
The black sum is left unaltered.
Thus, we get:
\eqn{
	{\sum_{|\mm|=1} \!\!
		a^i_\mm 
		\sum_{n=2}^\infty \sum_{|\nn|=n-1} \!\!\!\!
		\bm{\sigma}_\nn^\mm}
	+ \BLUE{\sum_{m=2}^\infty \sum_{|\mm|=m} \!\!\!
		a^i_\mm
		\sum_{n=\ORANGE{2}}^\infty \sum_{|\nn|={n-m}} \!\!\!\!\!
		\bm{\sigma}_\nn^\mm}
	+ \GREEN{\sum_{m=1}^\infty \sum_{|\mm|=m} \!\!\!
		\alpha^i_\mm \ast \!
		\sum_{n=\ORANGE{2}}^\infty \sum_{|\nn|={n-m-1}} \!\!\!\!\!\!
		\bm{\sigma}_\nn^\mm}
\fullstop
}
The advantage of this way of expressing the sums is that we can now interchange the summations over $m$ and $n$ to obtain:
\eqn{
	\ORANGE{\sum_{n=2}^\infty}
	\left\{
		{\sum_{|\mm|=1} \!
			a^i_\mm \!\!\!
			\sum_{|\nn|=n-1} \!\!\!\!
			\bm{\sigma}_\nn^\mm}
		+ \BLUE{\sum_{m=2}^\infty \sum_{|\mm|=m} \!\!\!
			a^i_\mm \!\!\!
			\sum_{|\nn|={n-m}} \!\!\!\!\!
			\bm{\sigma}_\nn^\mm}
		+ \GREEN{\sum_{m=1}^\infty \sum_{|\mm|=m} \!\!\!
			\alpha^i_\mm \ast  \!\!\!\!\!
			\sum_{|\nn|={n-m-1}} \!\!\!\!\!\!
			\bm{\sigma}_\nn^\mm}
	\right\}
\fullstop
}
Observe that the black sum fits well into the \BLUE{blue} sum over $m$ to give the $m=1$ term.
So we get:
\eqn{
	\sum_{n=2}^\infty \sum_{m=1}^\infty \sum_{|\mm|=m}
	\left\{
		\BLUE{	a^i_\mm \!\!\!
			\sum_{|\nn|={n-m}} \!\!\!\!\!
			\bm{\sigma}_\nn^\mm}
		+ \GREEN{
			\alpha^i_\mm \ast  \!\!\!\!\!
			\sum_{|\nn|={n-m-1}} \!\!\!\!\!\!
			\bm{\sigma}_\nn^\mm}
	\right\}
\fullstop
}
Finally, notice that both sums are empty for $m > n$, so we get:
\eqn{
	\sum_{n=2}^\infty \sum_{m=1}^{\ORANGE{n}} \sum_{|\mm|=m}
	\left\{
		\BLUE{	a^i_\mm \!\!\!
			\sum_{|\nn|={n-m}} \!\!\!\!\!
			\bm{\sigma}_\nn^\mm}
		+ \GREEN{
			\alpha^i_\mm \ast  \!\!\!\!\!
			\sum_{|\nn|={n-m-1}} \!\!\!\!\!\!
			\bm{\sigma}_\nn^\mm}
	\right\}
\fullstop
}
The sum over $m$ is precisely the expression inside the integral in \eqref{211124133914} defining $\sigma_n^i$.
This shows that $\sigma$ satisfies the integral equation \eqref{211124152233}.

\paragraph*{Step 4: Convergence.}
Now we show that $\sigma$ is a uniformly convergent series on $X_0 \times \Xi$ and therefore defines a holomorphic map $X_0 \times \Xi \to \Complex^\NN$.
In the process, we also establish the estimate \eqref{211214194634}.

Let $\BB, \CC, \LL > 0$ be such that for all $(x, \xi) \in X_0 \times \Xi$, all $i = 1, \ldots, \NN$, and all $\mm \in \Natural^\NN$,
\eqntag{\label{211207142232}
	\big| a^i_\mm (x) \big| \leq \rho_m \CC \BB^{m}
\qqtext{and}
	\big| \alpha^i_\mm (x, \xi) \big| \leq \rho_m \CC \BB^{m} e^{\LL |\xi|}
\fullstop{,}
}
where $m = |\mm|$ and $\rho_m$ is the normalisation constant \eqref{211209172628}.
We claim that there are constants $\DD, \MM > 0$ such that for all $(x, \xi) \in X_0 \times \Xi$ and all $n \in \Natural$,
\eqntag{\label{211207142237}
	\big| \sigma^i_n (x, \xi) \big| \leq \DD \MM^n \frac{|\xi|^n}{n!} e^{\LL |\xi|}
\fullstop
}
If we achieve \eqref{211207142237}, then the uniform convergence and the exponential estimate \eqref{211214194634} both follow at once because
\eqn{
	\big| \sigma^i (x, \xi) \big|
		\leq \sum_{n=0}^\infty \big| \sigma_n^i (x, \xi) \big|
		\leq \sum_{n=0}^\infty \DD \MM^n \frac{|\xi|^n}{n!} e^{\LL |\xi|}
		\leq \DD e^{(\MM + \LL) |\xi|}
\fullstop
}
To demonstrate \eqref{211207142237}, we proceed in two steps.
First, we construct a sequence of positive real numbers $\set{\MM_n}_{n=0}^\infty$ such that for all $n \in \Natural$ and all $(x, \xi) \in X_0 \times \Xi$,
\eqntag{\label{211215185429}
	\big| \sigma_n^i (x, \xi) \big| \leq \MM_n \frac{|\xi|^n}{n!} e^{\LL |\xi|}
\fullstop
}
We will then show that there are constants $\DD, \MM$ such that $\MM_n \leq \DD \MM^n$ for all $n$.

\paragraph*{Step 4.1: Construction of $\set{\MM_n}$.}
We can take $\MM_0 \coleq \CC$ and $\MM_1 \coleq \CC (1 + \BB \MM_0)$ because $\sigma_0^i = a^i_{\bm{0}}$ and
\eqns{
	\big| \sigma_1^i \big|
	&\leq \int_0^\xi 
		\left( 
			|\alpha^i_{\bm{0}}| 
			+ \sum_{|\mm| = 1} |a^i_\mm| |\sigma_0^\mm|
		\right) |\dd{t}|
	\leq \int_0^\xi
		\left( 
			\CC e^{\LL |t|}
			+ \CC^2 \BB \rho_1 \sum_{|\mm| = 1} 1
		\right) |\dd{t}|
\\
	&\leq \CC (1 + \BB \MM_0) \int_0^{|\xi|} e^{\LL s} \dd{s}
	\leq \CC (1 + \BB \MM_0) |\xi| e^{\LL |\xi|}
\fullstop{,}
}
where in the final step we used \autoref{180824194855}.
Now, let us assume that we have already constructed the constants $\MM_0, \ldots, \MM_{n-1}$ such that $|\sigma_k^i| \leq \MM_k \frac{|\xi|^k}{k!} e^{\LL |\xi|}$ for all $k = 0, \ldots, n-1$ and all $i = 1, \ldots, \NN$.
Then we use formula \eqref{211124133914} together with \autoref{180824194855} and \autoref{211205075846} in order to derive an estimate for $\sigma_n$.

First, let us write down an estimate for $\bm{\sigma}_\nn^\mm$ using formula \eqref{211207111746}.
Thanks to \autoref{211205075846}, we have for each $i = 1, \ldots, \NN$ and all $n_i, m_i$:
\eqn{
	\sum_{|\jj_i| = n_j}^{\jj_i \in \Natural^{m_i}}
			\Big|
				\sigma^i_{j_{i,1}} \ast \cdots \ast \sigma^i_{j_{i,m_i}}
			\Big|
	\leq \sum_{|\jj_i| = n_i}^{\jj_i \in \Natural^{m_i}}  \!\!
			\MM_{j_{i,1}} \cdots \MM_{j_{i,m_i}}
			\frac{|\xi|^{n_i + m_i - 1}}{(n_i + m_i - 1)!} e^{\LL |\xi|}
\fullstop
}
Then, for all $\nn,\mm \in \Natural^\NN$ with $|\nn| + |\mm| \geq 1$,
\eqntag{
	\big| \bm{\sigma}^\mm_\nn \big|
		\leq
			\bm{\MM}^\mm_\nn
			\frac{|\xi|^{|\nn| + |\mm| - 1}}{(|\nn| + |\mm| - 1)!} e^{\LL |\xi|}
\fullstop
}
where $\bm{\MM}^\mm_\nn$ is the shorthand introduced in \eqref{211209180714}.
Therefore, formula \eqref{211124133914} gives the following estimate:
\eqns{
	|\sigma_n^i|
	&\leq \int_0^\xi \sum_{m=1}^n \sum_{|\mm|=m}
		\left\{ |a^i_\mm| \!\! \sum_{|\nn|=n-m} \!\!\!
			\big| \bm{\sigma}_\nn^\mm \big|  
			+ \!\!\!\! \sum_{|\nn|=n-m-1} \!\!\!\!\!\!
			\big| \alpha_\mm^i \ast \bm{\sigma}_\nn^\mm \big|
		\right\} |\dd{t}|
\\	
	&\leq \sum_{m=1}^n \sum_{|\mm|=m}
		\left\{
			\rho_m \CC \BB^m \!\! \sum_{|\nn|=n-m} \!\!\! \bm{\MM}^\mm_\nn
			+ \rho_m \CC \BB^m 
				\!\!\!\!\! \sum_{|\nn|=n-m-1} \!\!\!\!\!\! \bm{\MM}^\mm_\nn
		\right\}
		\int_0^\xi \frac{|t|^{n - 1}}{(n - 1)!} e^{\LL |t|} |\dd{t}|
\\
	&\leq \sum_{m=1}^n
		\rho_m \CC \BB^m
		\sum_{|\mm|=m}
		\left\{
			\sum_{|\nn|=n-m} \!\!\! \bm{\MM}^\mm_\nn
			+ \!\!\! \sum_{|\nn|=n-m-1} \!\!\!\!\!\! \bm{\MM}^\mm_\nn
		\right\}
		\frac{|\xi|^n}{n!} e^{\LL |\xi|}
}
Thus, this expression allows us to define the constant $\MM_n$ for $n \geq 2$.
In fact, a quick glance at this formula reveals that it can be extended to $n = 0, 1$ by defining
\eqntag{\label{211207174849}
	\MM_n \coleq
		\sum_{m=0}^n
			\rho_m \CC \BB^m
		\sum_{|\mm|=m}
		\left\{
			\sum_{|\nn|=n-m} \!\!\! \bm{\MM}^\mm_\nn
			+ \!\!\! \sum_{|\nn|=n-m-1} \!\!\!\!\!\! \bm{\MM}^\mm_\nn
		\right\}
\qqquad
	\forall n \in \Natural
\fullstop
}
Indeed, if $m = 0$, then the two sums inside the brackets can only possibly be nonzero when $n = 0$, in which case the second sum is empty and the first sum is $1$, so we recover $\MM_0 = \CC$.
Likewise, if $n = 1$, then the $m = 0$ term is $0 + \CC$ and the $m=1$ term is $\CC \BB \MM_0 + 0$, so again we recover the constant $\MM_1$ defined previously.

\paragraph*{Step 4.2: Bounding $\MM_n$.}
To see that $\MM_n \leq \DD \MM^n$ for some $\DD, \MM > 0$, consider the following two power series in an abstract variable $t$:
\eqntag{\label{211207181737}
	\hat{p} (t) \coleq \sum_{n=0}^\infty \MM_n t^n
\qqtext{and}
	\QQ (t) 
		\coleq \sum_{m=0}^\infty \CC \BB^m t^m
\fullstop
}
Notice that $\QQ (t)$ is convergent and $\QQ (0) = \CC = \MM_0$.
We will show that $\hat{p} (t)$ is also a convergent power series.
The key observation is that $\hat{p}$ satisfies the following functional equation:
\eqntag{\label{211207181733}
	\hat{p} (t) = (1+t) \QQ \big( t \hat{p} (t) \big)
\fullstop
}
This equation was found by trial and error.
In order to verify it, we rewrite the power series $\QQ (t)$ in the following way:
\eqntag{
	\QQ (t)	= \sum_{m=0}^\infty \sum_{|\mm|=m} \!\!
			\rho_m \CC \BB^m t^\mm
\fullstop
}
Then \eqref{211207181733} is straightforward to verify by direct substitution and comparing the coefficients of $t^n$ using the defining formula \eqref{211207174849} for $\MM_n$.
Thus, the righthand side of \eqref{211207181733} expands as follows:
\eqns{
	&\phantom{=}~~
	(1+t) \sum_{m=0}^\infty \sum_{|\mm|=m} \!\!
		\rho_m \CC \BB^m 
		\left( t \sum_{n=0}^\infty \MM_n t^n \right)^{\!\!\mm}
\\
	&= (1+t) \sum_{m=0}^\infty \sum_{|\mm|=m} \!\!
		\rho_m \CC \BB^m t^m
		\left( \sum_{n_1=0}^\infty \MM_{n_1} t^{n_1} \right)^{\!\!m_1} \!\!\!\!
		\cdots
		\left( \sum_{n_\NN=0}^\infty \MM_{n_\NN} t^{n_\NN} \right)^{\!\!m_\NN}
\\
	&= (1+t) \sum_{m=0}^\infty \sum_{|\mm|=m} \!\!
		\rho_m \CC \BB^m
		\sum_{n=0}^\infty \sum_{|\nn|=n} \bm{\MM}^\mm_\nn t^{n+m}
\\
	&= (1+t) \sum_{m=0}^\infty \sum_{|\mm|=m} \!\!
		\rho_m \CC \BB^m
		\sum_{n=0}^\infty
		\sum_{|\nn|=\ORANGE{n-m}} \bm{\MM}^\mm_\nn t^{\ORANGE{n}}
\\
	&= \sum_{n=0}^\infty \sum_{m=0}^\infty
		\rho_m \CC \BB^m \!\!
		\sum_{|\mm|=m}
		\left\{
			\sum_{|\nn|=n-m} \bm{\MM}^\mm_\nn t^{n}
			+ \sum_{|\nn|=n-m} \bm{\MM}^\mm_\nn t^{n+1}
		\right\}
\\
	&= \sum_{n=0}^\infty \sum_{m=0}^{\ORANGE{n}}
		\rho_m \CC \BB^m \!\!
		\sum_{|\mm|=m}
		\left\{
			\sum_{|\nn|=n-m} \bm{\MM}^\mm_\nn
			+ \sum_{|\nn|=\ORANGE{n-m-1}} \bm{\MM}^\mm_\nn
		\right\}
		t^{n}
}
In the final equality, we once again noticed that both sums inside the curly brackets are zero whenever $m > n$.

Now, consider the following holomorphic function in two variables $(t,p)$:
\eqntag{
	\FF (t, p) \coleq - p + (1+t) \QQ (tp)
\fullstop
}
It has the following properties:
\eqn{
	\FF (0, \CC) = 0
\qqtext{and}
	\evat{\frac{\del \PP}{\del p}}{(t,p) = (0, \CC)} = - 1 \neq 0
\fullstop
}
By the Holomorphic Implicit Function Theorem, there exists a unique holomorphic function $p(t)$ near $t = 0$ such that $p(0) = \CC$ and $\FF \big(t, p(t)\big) = 0$.
Therefore, $\hat{p} (t)$ must be the convergent Taylor series expansion of $p(t)$ at $t = 0$, so its coefficients grow at most exponentially: i.e., there are constants $\DD, \MM > 0$ such that $\MM_n \leq \DD \MM^n$.
This completes the proof of the Main Technical Claim and hence of \autoref{211124143506}.
\end{proof}

At last, we are able to collect all our work in order to finish the proof of the Gevrey Asymptotic Implicit Function Theorem (\autoref{211125100306}).

\begin{proof}[Proof of \autoref{211125100306}]
\label{211215172900}
By the Formal Implicit Function Theorem (\autoref{211209161918}), there is a subdomain $X_0 \subset X$ containing $x_0$ such that the equation $\FF (x, \hbar, z) = 0$ has a unique formal solution $\hat{f}$ satisfying $f_0 (x_0) = z_0$.
Let $f_0, f_1$ be its leading- and next-to-leading-order parts in $\hbar$.
As in \autoref{211208172717}, we change variables as $z = f_0 + \hbar f_1 + \hbar w$ to transform the equation $\FF (x, \hbar, z) = 0$ into $w = \hbar \GG (x, \hbar, w)$.
By \autoref{211124143506}, this equation has a unique holomorphic solution $w = g(x,\hbar)$ on $X_0 \times S_0$ for some sectorial subdomain $S_0 \subset S$ still with opening $A$ and admitting a uniform Gevrey asymptotic expansion as $\hbar \to 0$ along $\bar{A}$.
Finally, we define $f \coleq f_0 + \hbar f_1 + \hbar g$ which is readily seen to have all the desired properties.
\end{proof}

\section{Linear Algebra in Gevrey Asymptotic Families}
\label{211209190601}

In this section, we provide an application of our main theorem to the study of asymptotic families of holomorphic matrices over a sector.
Namely, we prove the following theorem, from which \autoref{211209190924} follows immediately as a special case.

\begin{thm}[{Jordan Decomposition in Gevrey Asymptotic Families}]{211210093539}
\mbox{}\\
Fix a domain $X \subset \Complex^d_x$ and a point $x_0 \in X$.
Let $S \subset \Complex_\hbar$ be a sectorial domain at the origin and opening arc $A$ with opening angle $|A| = \pi$.
Let $\AA = \AA (x, \hbar)$ be a holomorphic $n\!\!\times\!\!n$-matrix on $X \times S$ which admits a uniform Gevrey asymptotic expansion
\vspace{-5pt}
\eqntag{\label{211210151235}
	\AA (x, \hbar) \simeq \hat{\AA} (x, \hbar)
\quad
\text{as $\hbar \to 0$ along $\bar{A}$, unif. $\forall x \in X$\fullstop}
}
Suppose that the distinct eigenvalues of its $\hbar$-leading-order part $\AA_{00} \coleq \AA_0 (x_0)$ at the point $x_0$ are $a_1, \ldots, a_m \in \Complex$ with respective multiplicities $n_1, \ldots, n_m$.
Let $\PP_{00}$ be a constant invertible $n\!\!\times\!\!n$-matrix that puts $\AA_{00}$ into a Jordan normal form:
\eqn{
	\PP_{00} \AA_{00} \PP_{00}^{-1} 
		= \Lambda_{00} 
		\coleq \diag \big( a_1 \II_{n_1} + \NN_1, \ldots, a_m \II_{n_m} + \NN_m \big)
\fullstop{,}
}
where $\II_{n_i}$ is the identity $n_i\!\!\times\!\!n_i$-matrix and $\NN_i$ is a nilpotent $n_i\!\!\times\!\!n_i$-matrix containing zeros in all positions except those in the first superdiagonal, which may contain either zeros or ones.

Then there is a subdomain $X_0 \subset X$ containing $x_0$ and a sectorial subdomain $S_0 \subset S$ with the same opening $A$ such that there is an invertible $n\!\!\times\!\!n$-matrix $\PP = \PP (x, \hbar)$ on $X_0 \times S_0$ that admits a uniform Gevrey asymptotic expansion
\eqntag{\label{211210140824}
	\PP (x, \hbar) \simeq \hat{\PP} (x, \hbar) = \sum_{k=0}^\infty \PP_{k} (x) \hbar^k
\quad
\text{as $\hbar \to 0$ along $\bar{A}$, unif. $\forall x \in X_0$\fullstop{,}}
}
such that $\PP_0 (x_0) = \PP_{00}$ and which block-diagonalises the matrix $\AA$:
\eqntag{\label{211210140828}
	\PP \AA \PP^{-1} = \Lambda = \diag \big( \Lambda_1, \ldots, \Lambda_m \big)
\fullstop{,}
}
where each $\Lambda_i = \Lambda_i (x, \hbar)$ is an $n_i \!\times\! n_i$-matrix which admits a uniform Gevrey asymptotic expansion
\eqntag{
	\Lambda_i (x, \hbar) \simeq \hat{\Lambda}_i (x, \hbar) = \sum_{k=0}^\infty \Lambda_{i,k} (x) \hbar^k
\quad
\text{as $\hbar \to 0$ along $\bar{A}$, unif. $\forall x \in X_0$\fullstop{,}}
}
with $\Lambda_{i,0} (x_0) = a_i \II_{n_1} + \NN_i$.
Furthermore, the transformation $\PP$ is the uniform Borel resummation of its asymptotic power series $\hat{\PP}$ in the direction $\theta$ that bisects the arc $A$: for all $(x,\hbar) \in X_0 \times S_0$,
\eqntag{\label{211210151220}
	\PP = \cal{S}_\theta \big[ \, \hat{\PP} \, \big]
\fullstop
}
\end{thm}

\begin{proof}
Using standard theory (see, e.g., \cite[\S25.2]{MR0460820}), we can find a holomorphic invertible matrix $\PP_0 (x)$ on a domain $X_0 \subset X$ such that
\eqn{
	\PP_0 \AA_0 \PP_0^{-1} = \diag (\Lambda_{1,0}, \ldots, \Lambda_{m,0})
}
and $\PP_0 (x_0) = \PP_{00}$ where each $\Lambda_{i,0} = \Lambda_{i,0} (x)$ is a holomorphic $n_i\!\times\!n_i$-matrix on $X_0$ with the property that $\Lambda_{i,0} (x_0) = a_i \II_{n_1} + \NN_i$.
To simplify notation, let us assume that the leading-order matrix $\AA_{0}$ has already been diagonalised over $X_0$, so
\eqn{\label{211210145727}
	\PP_0 = \II
\qtext{and}
	\AA_{0} = \diag \big( \Lambda_{1,0}, \ldots, \Lambda_{m,0} \big) = \Lambda_0
\fullstop
}
Our goal is to find a holomorphic matrix $\PP = \PP (x, \hbar)$ whose leading order is $\PP_0$ and a holomorphic block-diagonal matrix $\Lambda = \diag(\Lambda_1, \ldots, \Lambda_m)$, with blocks $\Lambda_{i} = \Lambda_{i} (x, \hbar)$ of size $n_i\!\times\!n_i$, whose leading order is $\Lambda_0$, such that
\eqntag{\label{211210150242}
	\PP \AA = \Lambda \PP
\fullstop
}

Let us break up the eigenvalues $a_1, \ldots, a_m$ arbitrarily into two separate groups.
Pick any $p \in \set{1, \ldots, m}$ and let
\eqntag{\label{211210150532}
	\Lambda'_0 \coleq \diag (\Lambda_{1,0}, \ldots, \Lambda_{p,0})
\qtext{and}
	\Lambda''_0 \coleq \diag (\Lambda_{p+1,0}, \ldots, \Lambda_{m,0})
\fullstop
}
Note that $\Lambda'_0$ and $\Lambda''_0$ have no eignvalues in common.
Put $n' \coleq n_1 + \ldots + n_p$ and $n'' \coleq n_{p+1} + \ldots + n_m$.
Then we block-partition the matrices $\Lambda$ and $\AA$ accordingly:
\eqntag{\label{211210150529}
	\Lambda =
		\left[
			\begin{array}{c|c}
			\Lambda' & 0 \\
			\hline
			0 & \Lambda''
			\end{array}
		\right]
\qqtext{and}
	\AA 
	=
		\left[
			\begin{array}{c|c}
			\AA_{11} & \AA_{12} \\
			\hline
			\AA_{21} & \AA_{22}
			\end{array}
		\right]
	=
		\left[
			\begin{array}{c|c}
			n' \times n' & n' \times n'' \\
			\hline
			n'' \times n' & n'' \times n''
			\end{array}
		\right]
\fullstop{,}
}
where we have indicated the sizes of the blocks $\AA_{ij}$.
Inspired by techniques in \cite{MR0096016,MR0214869,MR0245929}, we search for $\PP$ in the following block-matrix form:
\eqntag{\label{211210151144}
	\PP =
		\left[
			\begin{array}{c|c}
			\II_{n'} & \SS \\
			\hline
			\TT & \II_{n''}
			\end{array}
		\right]
	=
		\left[
			\begin{array}{c|c}
			n' \times n' & n' \times n'' \\
			\hline
			n'' \times n' & n'' \times n''
			\end{array}
		\right]
\fullstop{,}
}
where $\II_{n'}, \II_{n''}$ are respectively the identity $n'\!\times\!n'$- and $n''\!\times\!n''$-matrices.
Substituting the block-partitions \eqref{211210150529} and the ansatz \eqref{211210151144} into equation \eqref{211210150242} yields four conditions:
\eqntag{\label{211210151125}
\begin{aligned}
	\AA_{11} + \SS \AA_{21} &= \Lambda'
\fullstop{;}
&	\qqquad
	\AA_{12} + \SS \AA_{22} &=  \Lambda' \SS
\fullstop{;}
\\	\AA_{22} + \TT \AA_{12} &= \Lambda''
\fullstop{;}
&	\qqquad
	\AA_{21} + \TT \AA_{11} &= \Lambda'' \TT
\fullstop
\end{aligned}
}
Matrices $\Lambda', \Lambda''$ can be eliminated from the two equations on the right.
This leads to two uncoupled matrix quadratic equations for $\SS$ and $\TT$:
\eqntag{\label{211210151851}
	\AA_{12} + \SS \AA_{22} - \AA_{11} \SS - \SS \AA_{21} \SS = 0
\qtext{and}
	\AA_{21} + \TT \AA_{11} - \AA_{22} \TT - \SS \AA_{12} \TT = 0
~\GREY{.}
}
Let us focus on solving the equation for $\SS$.
Observe that its leading-order in $\hbar$ is simply
\eqntag{\label{211210151909}
	\SS_0 \Lambda''_0 - \Lambda'_0 \SS_0 = 0
\fullstop
}
A simple but remarkable fact from linear algebra (see, e.g., \cite[Theorem 4.1]{MR0460820}) says that this equation possess solutions other than $\SS_0 = 0$ if and only if $\Lambda'_0$ and $\Lambda''_0$ have at least one eigenvalue in common, which is contrary to how the matrices $\Lambda'_0$ and $\Lambda''_0$ were defined.
Thus, $\SS_0 = 0$.

Now, put $\NN \coleq n'n''$, and let $w = (w_1, \ldots, w_\NN)$ be the $\NN$-dimensional vector whose components are the entries of $\SS$ in some order.
Then the quadratic equation \eqref{211210151851} for $\SS$ can be written in the form
\eqntag{\label{211210152151}
	w = \hbar \GG (x, \hbar, w)
}
where $\GG$ is a holomorphic map $X_0 \times S \times \Complex_w^\NN \to \Complex^\NN$ which is quadratic in the components of $w$.
By the Gevrey Asymptotic Implicit Function Theorem (\autoref{211125100306}) (or more specifically by \autoref{211124143506}), there is a sectorial subdomain $S_0 \subset S$ with the same opening $A$ such that there is a unique holomorphic map $g : X_0 \times S_0 \to \Complex^\NN$ which admits a uniform Gevrey asymptotic expansion as $\hbar \to 0$ along $\bar{A}$ and such that $g (x,\hbar) = \hbar \GG \big(x, \hbar, g(x,\hbar) \big) = 0$.
This implies the existence and uniqueness of a holomorphic matrix $\SS = \SS (x, \hbar)$ on $X_0 \times S_0$ satisfying \eqref{211210151851} and admitting a uniform Gevrey asymptotic expansion as $\hbar \to 0$ along $\bar{A}$.

Using exactly the same argument, we can derive a unique solution $\TT = \TT (x, \hbar)$ of \eqref{211210151851} on $X_0 \times S_0$ at the expense of only possibly having to shrink the radial size of $S_0$ (but not the opening $A$).
As a result, we have found a unique transformation $\PP$ on $X_0 \times S_0$ defined by \eqref{211210151144} and a unique block-diagonal matrix $\Lambda = \diag (\Lambda', \Lambda'')$ on $X_0 \times S_0$ defined by the two equations on the left in \eqref{211210151125}, which satisfy \eqref{211210150242} and have the desired asymptotic properties.

If the total number of distinct eigenvalues $m$ is $2$, then we can proceed no further: $\Lambda, \PP$ are the desired matrices for the assertions of the theorem.
Otherwise, the procedure outlined above should now be iterated to finally find the matrix $\PP$ that brings $\AA$ to the desired form $\Lambda$ from \eqref{211210140828}.
For example, in the next step after we have found the matrices $\Lambda', \Lambda''$ above, we may search for a transformation that further block-diagonalises $\Lambda'$ (that is, if $\Lambda'_0$ has at least two distinct eigenvalues; otherwise, proceed to block-diagonalise $\Lambda''$).
We may break the eigenvalues $a_1, \ldots, a_p$ up further into two groups and transform $\Lambda$ from \eqref{211210150529} to a block-diagonal matrix $\tilde{\Lambda} = \diag (\tilde{\Lambda}', \tilde{\Lambda}'', \Lambda'')$ by searching for a transformation in the form
\eqn{
	\left[
	\begin{array}{cc|c}
		\II & \multicolumn{1}{|c|}{\SS} & \\
		\cline{1-2}
	 	\TT & \multicolumn{1}{|c|}{\II} & \\
	 	\hline
	 	& & \II
	\end{array}
	\right]
\fullstop
\tag*{\qedhere}
}
\end{proof}

\begin{rem}{211214154402}
Under the assumptions of \autoref{211209190924}, we can use a slightly different argument to the one used to prove \autoref{211210093539} in order to show that the eigenvalues $\lambda_1, \ldots, \lambda_n$ of $\AA$ are well-defined and have the asserted properties.
Indeed, the characteristic polynomial of $\AA$ is a holomorphic function $\FF = \FF (x, \hbar, z)$ on $X \times S \times \Complex_z$, which is a polynomial in $z$ whose coefficients admit a uniform Gevrey asymptotic expansion as $\hbar \to 0$ along $\bar{A}$.
Its leading-order part $\FF_0 = \FF_0 (x, z)$ is the characteristic polynomial of the leading-order part $\AA_0$, so $\FF_0 (x_0, a_i) = 0$ for each $i = 1, \ldots, n$.
By assumption, the eigenvalues of $\AA_{00}$ are distinct, which means the discriminant of the polynomial $\FF_0$ at $x = a$ is nonzero.
As a result, the derivative $\del \FF_0 / \del z$ is nonzero at each $(x_0, a_i)$.
By the Gevrey Asymptotic Implicit Function Theorem (\autoref{211125100306}, or more specifically \autoref{211217121003}), there is a subdomain $X_0 \subset X$ containing $x_0$ and a sectorial subdomain $S_0 \subset S$ with the same opening $A$ such that there are unique holomorphic functions $\lambda_i = \lambda_i (x, \hbar)$ on $X_0 \times S_0$ that admit uniform Gevrey asymptotic expansions \eqref{211209195205} with $\hbar$-leading-orders satisfying $\lambda_{i,0} (x_0) = a_i$.	
\end{rem}

\begin{appendices}
\appendixsectionformat

\section{Background Information}
\label{211215112252}

\subsection{Gevrey Asymptotics}
\label{211215123326}

\paragraph{}
A \dfn{sectorial domain} at the origin in $\Complex_\hbar$ is a simply connected domain $S \subset \Complex_\hbar^\ast = \Complex_\hbar \setminus \set{0}$ whose closure $\bar{S}$ in the real-oriented blowup $[\Complex_\hbar : 0]$ intersects the boundary circle $\Sphere^1$ in a closed arc $\bar{A} \subset \Sphere^1$ with nonzero length.
The open arc $A$ is called the \dfn{opening} of $S$, and its length $|A|$ is called the \dfn{opening angle} of $S$.
A \dfn{Borel disc} of \dfn{diameter} $\RR >0$ is the sectorial domain $S = \set{ \hbar \in \Complex_\hbar ~\big|~ \Re (1/\hbar) > 1/\RR }$.
Its opening is $A = (-\tfrac{\pi}{2}, + \tfrac{\pi}{2})$.
Likewise, a Borel disc bisected by a direction $\theta \in \Sphere^1$ is the sectorial domain $S = \set{ \hbar \in \Complex_\hbar ~\big|~ \Re (e^{i\theta}/\hbar) > 1/\RR }$.
Its opening is $A = (\theta -\tfrac{\pi}{2}, \theta + \tfrac{\pi}{2})$.

\paragraph{}
A holomorphic function $f (\hbar)$ on a sectorial domain $S$ is admits a power series $\hat{f} (\hbar)$ as its \dfn{asymptotic expansion as $\hbar \to 0$ along $A$} (or \dfn{as $\hbar \to 0$ in $S$}) if, for every $n \geq 0$ and every compactly contained subarc $A_0 \Subset A$, there is a sectorial subdomain $S_0 \subset S$ with opening $A_0$ and a real constant $\CC_{n,0} > 0$ such that
\eqntag{\label{200720153758}
	\left| f(\hbar) - \sum_{k=0}^{n-1} f_k \hbar^k \right| \leq \CC_{n,0} |\hbar|^n
}
for all $\hbar \in S_0$.
The constants $\CC_{n,0}$ may depend on $n$ and the opening $A_0$.
If this is the case, we write
\eqntag{\label{200720175735}
	f (\hbar) \sim \hat{f} (\hbar)
\qqqquad
	\text{as $\hbar \to 0$ along $A$\fullstop}
}
If the constants $\CC_{n,0}$ in \eqref{200720153758} can be chosen uniformly for all compactly contained subarcs $A_0 \Subset A$ (i.e., independent of $A_0$ so that $\CC_{n,0} = \CC_n$ for all $n$), then we write
\eqntag{\label{210220160756}
	f (\hbar) \sim \hat{f} (\hbar)
\qqqquad
	\text{as $\hbar \to 0$ along $\bar{A}$\fullstop}
}

\paragraph{}
We also say that the holomorphic function $f$ admits $\hat{f}$ as its \dfn{Gevrey asymptotic expansion as $\hbar \to 0$ along $A$} if the constants $\CC_{n,0}$ in \eqref{200720153758} depend on $n$ like $\CC_0 \MM_0^n n!$.
More explicitly, for every compactly contained subarc $A_0 \Subset A$, there is a sectorial domain $S_0 \subset S$ with opening $A_0 \Subset A$ and real constants $\CC_0, \MM_0 > 0$ which give the bounds
\eqntag{\label{200722160158}
	\left| f(\hbar) - \sum_{k=0}^{n-1} f_k \hbar^k \right| \leq \CC_0 \MM_0^n n! |\hbar|^n
}
for all $\hbar \in S_0$ and all $n \geq 0$.
In this case, we write
\eqntag{\label{210225131044}
	f (\hbar) \simeq \hat{f} (\hbar)
\qqqquad
	\text{as $\hbar \to 0$ along $A$\fullstop}
}
If in addition to \eqref{200722160158}, the constants $\CC_0, \MM_0$ can be chosen uniformly for all $A_0 \Subset A$, then we will write
\eqntag{\label{210225134416}
	f (\hbar) \simeq \hat{f} (\hbar)
\qqqquad
	\text{as $\hbar \to 0$ along $\bar{A}$\fullstop}
}

\paragraph{}
A formal power series $\hat{f} (\hbar) = \sum f_n \hbar^n$ is a \dfn{Gevrey power series} if there are constants $\CC, \MM > 0$ such that for all $n \geq 0$,
\eqntag{\label{200723182724}
	| f_n | \leq \CC \MM^n n!
\fullstop
}

\paragraph{}
All the above definitions translate immediately to cover vector-valued holomorphic functions on $S$ by using, say, the Euclidean norm in all the above estimates.

\subsection{Borel-Laplace Theory}
\label{211215123550}

\paragraph{}
\label{211215140949}
Let $\Xi_\theta \coleq \set{ \xi \in \Complex_\xi ~\big|~ \op{dist} (\xi, e^{i\theta}\Real_+) < \epsilon}$, where $e^{i \theta} \Real_+$ is the real ray in the direction $\theta$.
Let $\phi = \phi (\xi)$ be a holomorphic function on $\Xi_\theta$.
Its \dfn{Laplace transform} in the direction $\theta$ is defined by the formula:
\eqntag{\label{200624181217}
	\Laplace_\theta [\, \phi \,] (x, \hbar)
		\coleq \int\nolimits_{e^{i\theta} \Real_+} \phi (x, \xi) e^{-\xi/\hbar} \dd{\xi}
\fullstop
}
When $\theta = 0$, we write simply $\Laplace$.
Clearly, $\phi$ is Laplace-transformable in the direction $\theta$ if $\phi$ has \dfn{at-most-exponential growth} as $|\xi| \to + \infty$ along the ray $e^{i\theta} \Real_+$.
Explicitly, this means there are constants $\AA, \LL > 0$ such that for all $\xi \in \Xi_\theta$,
\eqntag{\label{211215140908}
	\big| \phi (\xi) \big| \leq \AA e^{\LL |\xi|}
\fullstop
}

\paragraph{}
The convolution product of two holomorphic functions $\phi, \psi$ is defined by the following formula:
\eqntag{
	\phi \ast \psi (\xi)
	\coleq 
	\int\nolimits_0^\xi \phi (\xi - y) \psi (y) \dd{y}
\fullstop{,}
}
where the path of integration is a straight line segment from $0$ to $\xi$.

\paragraph{}
Let $f$ be a holomorphic function on a Borel disc $S = \set{ \hbar \in \Complex_\hbar ~\big|~ \Re (e^{i\theta}/\hbar) > 1/\RR }$.
The (analytic) \dfn{Borel transform} (a.k.a., the \dfn{inverse Laplace transform}) of $f$ in the direction $\theta$ is defined by the following formula:
\eqntag{\label{210617101748}
	\Borel_\theta [\, f \,] (x, \xi)
		\coleq \frac{1}{2\pi i} \oint\nolimits_\theta f(x, \hbar) e^{\xi / \hbar} \frac{\dd{\hbar}}{\hbar^2}
\fullstop{,}
}
where the integral is taken along the boundary of any Borel disc
\eqn{
	S' = \set{ \hbar \in \Complex_\hbar ~\big|~ \Re (e^{i\theta}/\hbar) > 1/\RR' } \subset S
}
of strictly smaller diameter $\RR' < \RR$, traversed anticlockwise (i.e., emanating from the singular point $\hbar = 0$ in the direction $\theta - \pi/2$ and reentering in the direction $\theta + \pi/2$).
When $\theta = 0$, we write simply $\Borel$.

The fundamental fact that connects Gevrey asymptotics and the Borel transform is the following (cf. \cite[Lemma B.5]{MY2008.06492}).
If $f = f(\hbar)$ is a holomorphic function defined on a sectorial domain $S$ with opening angle $|A| = \pi$ and $f$ admits Gevrey asymptotics as $\hbar \to 0$ along the \textit{closed} arc $\bar{A}$, then the analytic Borel transform $\phi (\xi) = \Borel_\theta [f] (\xi)$ defines a holomorphic function on a tubular neighbourhood $\Xi_\theta$ of some thickness $\epsilon > 0$.
Moreover, its Laplace transform in the direction $\theta$ is well-defined and satisfies $\Laplace_\theta [\phi] = f$.

\paragraph{}
Similarly, for a power series $\hat{f} (\hbar)$, the (formal) \dfn{Borel transform} is defined by
\eqntag{
	\hat{\phi} (\xi) =
	\hat{\Borel} [ \, \hat{f} \, ] (\xi)
		\coleq \sum_{k=0}^\infty \phi_k \xi^k
\qtext{where}
	\phi_k \coleq \tfrac{1}{k!} f_{k+1}
\fullstop
}
The fundamental fact that connects Gevrey power series and the formal Borel transform is the following (cf. \cite[Lemma B.8]{MY2008.06492}).
If $\hat{f}$ is a Gevrey power series, then its formal Borel transform $\hat{\phi}$ is a convergent power series in $\xi$.
Furthermore, a Gevrey power series $\hat{f} (\hbar)$ is called a \dfn{Borel summable series} in the direction $\theta$ if its convergent Borel transform $\hat{\phi} (\xi)$ admits an analytic continuation $\phi (\xi) = \rm{AnCont}_\theta [\, \hat{\phi} \,] (\xi)$ to a tubular neighbourhood $\Xi_\theta$ of the ray $e^{i\theta} \Real_+$ with at-most-exponential growth in $\xi$ at infinity in $\Xi_\theta$.
If this is the case, the Laplace transform $\Laplace_\theta [\phi] (\hbar)$ is well-defined and defines a holomorphic function $f(\hbar)$ on some Borel disc $S$ bisected by the direction $\theta$, and we say that $f(\hbar)$ is the \dfn{Borel resummation} in direction $\theta$ of the formal power series $\hat{f} (\hbar)$, and we write
\eqn{
	f (\hbar) = \cal{S}_\theta \big[ \, \hat{f} (\hbar) \, \big] (\hbar)
\fullstop
}
If $\theta = 0$, we write simply $\cal{S}$.
Expressly, we have the following formulas:
\eqn{
	\cal{S}_\theta \big[ \, \hat{f} (\hbar) \, \big] (\hbar)
	= \Laplace_\theta \big[ \, \phi \, \big] (\hbar)
	= \Laplace_\theta \big[ \, \rm{AnCont}_\theta [\, \hat{\phi} \,] \, \big] (\hbar)
\fullstop
}
Thus, Borel resummation $\cal{S}_\theta$ can be seen as a map from the set of (germs of) holomorphic functions $f$ on $S$ with $|A| = \pi$ satisfying \eqref{210225134416} to the set of Borel summable power series.
One of the most fundamental theorems in Gevrey asymptotics and Borel-Laplace theory is a theorem of Nevanlinna \cite[pp.44-45]{nevanlinna1918theorie}\footnote{It was rediscovered and clarified decades later by Sokal \cite{MR558468}; see also \cite[p.182]{zbMATH00797135}, \cite[Theorem 5.3.9]{MR3495546}, as well as \cite[§B.3]{MY2008.06492}.}, which says that this map $\cal{S}_\theta$ is invertible and its inverse is the asymptotic expansion $\ae$.

\subsection{Some Useful Elementary Estimates}

Here, for reference, we collect some elementary estimates used in this paper.
Their proofs are straightforward (see \cite[Appendix C.4]{MY2008.06492}).

\begin{lem}{180824194855}
For any $\RR \geq 0$, any $\LL \geq 0$, and any nonnegative integer $n$,
\eqn{
	\int_0^{\RR} \frac{r^n}{n!} e^{\LL r} \dd{r}
		\leq \frac{\RR^{n+1}}{(n+1)!} e^{\LL \RR}
\fullstop
}	
\end{lem}

\begin{lem}{211205075846}
Let $i_1, \ldots, i_m$ be nonnegative integers and put $n \coleq i_1 + \cdots + i_m$.
Let $f_{i_1}, \ldots, f_{i_j}$ be holomorphic functions on $\Xi \coleq \set{\xi ~\big|~ \op{dist} (\xi, \Real_+) < \epsilon}$ for some $\epsilon > 0$.
If there are constants $\MM_{i_1}, \ldots, \MM_{i_m}, \LL \geq 0$ such that
\eqn{
	\big| f_{i_k} (\xi) \big| \leq \MM_{i_k} \frac{|\xi|^{i_k}}{i_k!} e^{\LL |\xi|}
\rlap{\qquad $\forall \xi \in \Xi$\fullstop{,}}
}
then their total convolution product satisfies the following bound:
\eqn{
	\big| f_{i_1} \ast \cdots \ast f_{i_m} (\xi) \big| 
		\leq \MM_{i_1} \cdots \MM_{i_m} \frac{|\xi|^{i+m-1}}{(i+m-1)!} e^{\LL |\xi|}
\rlap{\qqquad $\forall \xi \in \Xi$\fullstop}
}
\end{lem}

\section[Proof of \autoref*{211125100306}: Scalar Case]
{Proof of \autoref{211125100306}: Scalar Case}
\label{211125162208}

This section is dedicated to proving our main result, the Gevrey Asymptotic Implicit Function Theorem (\autoref{211125100306}), in the scalar case, $\NN = 1$.
The overall strategy of the proof is to first construct a formal solution $z = \hat{f}$ of the equation $\FF (x, \hbar, z) = 0$ using the ordinary Holomorphic Implicit Function Theorem at the leading-order in $\hbar$ and then use a recursion to determine all higher-order corrections.
We then want to apply the Borel resummation to $\hat{f}$ to get $f$.
To do so, we first make a convenient change of variables $z \mapsto w$ in order to put our equation into a certain standard form which is more amenable to the Borel transform.
Applying the Borel transform, we obtain a first-order ordinary differential equation for $\sigma = \Borel [ w ]$, albeit nonlinear and with convolution.
Nevertheless, this ODE is easy to convert into an integral equation, which we then proceed to solve using the method of successive approximations.
To show that this sequence of approximations converges to an actual solution $\sigma$, we give an estimate on the terms of this sequence by employing in an interesting way the ordinary Holomorphic Implicit Function Theorem.
This estimate also allows us to conclude that the Laplace transform $g = \Laplace [\sigma]$ of the obtained solution $\sigma$ exists and defines a holomorphic solution of our equation in standard form.
Undoing the change of variables $z \mapsto w$ sends $g$ to the desired solution $f$.

\subsection{Formal Perturbation Theory}

The starting point is the following classical result whose proof is supplied below for completeness.

\enlargethispage{20pt}
\begin{prop}[Formal Implicit Function Theorem]{211126190846}
\mbox{}\newline
Fix a domain $X \subset \Complex^d_x$ and a point $(x_0, z_0) \in X \times \Complex_z$.
Let
\eqntag{\label{211125163516}
	\hat{\FF} = \hat{\FF} (x,\hbar,z) = \sum_{k=0}^\infty \FF_k (x,z) \hbar^k
}
be a formal power series in $\hbar$ whose coefficients $\FF_k$ are holomorphic functions on $X \times \Complex_z$ such that $\FF_0 (x_0, z_0) = 0$ and the derivative $\del \FF_0 \big/ \del z$ is nonzero at $(x_0, z_0)$.

Then there is subdomain $X_0 \subset X$ containing $x_0$ such that there is a unique formal power series
\eqntag{\label{211125163728}
	\hat{f} = \hat{f} (x,\hbar) = \sum_{n=0}^\infty f_n (x) \hbar^n
}
whose coefficients $f_n$ are holomorphic functions on $X_0$, satisfying 
\eqntag{
	f_0 (x_0) = z_0
\qqtext{and}
	\hat{\FF} \big( x, \hbar, \hat{f} (x,\hbar) \big) = 0
\qquad
	\text{$\forall x \in X_0$\fullstop}
}
In other words, the equation $\hat{\FF} (x, \hbar, z) = 0$ has a unique formal power series solution $z = \hat{f}$ defined near the point $x_0$ and satisfying $f_0 (x_0) = z_0$.
In fact, all the higher-order coefficients $f_k$ are uniquely determined by $f_0$.

In particular, if $S \subset \Complex_\hbar$ is a sectorial domain at the origin, and $\FF$ is a holomorphic function on $X \times S \times \Complex_z$ which admits the power series $\hat{\FF}$ as a uniform asymptotic expansion as $\hbar \to 0$ in $S$, then the equation $\FF (x,\hbar,z) = 0$ has a unique formal power series solution $z = \hat{f}$ near $x_0$ such that $f_0 (x_0) = z_0$.
\end{prop}

\begin{proof}
The proof is a calculation that amounts to plugging the solution ansatz into the equation solving it order-by-order in $\hbar$.
Write the double power series expansion of $\hat{\FF}$ in $(\hbar, z)$ as
\eqntag{\label{211126144357}
	\hat{\FF} (x, \hbar, z) = \sum_{k=0}^\infty \sum_{m=0}^\infty \FF_{km} (x) \hbar^k z^m
\fullstop
}

\textsc{Step 1: Expand order-by-order.}
Plugging the solution ansatz \eqref{211125163728} into the equation $\hat{\FF} (x, \hbar, z) = 0$ yields:
\eqntag{\label{211204115845}
	\sum_{n=0}^\infty \sum_{m=0}^\infty \sum_{k=0}^n 
		\sum_{i_1 + \ldots + i_m = n-k}
		\FF_{km} f_{i_1} \cdots f_{i_m} \hbar^n
	= 0
\fullstop
}
Now we solve it order-by-order in $\hbar$.

\textsc{Step 2: Leading-order part.}
At order $n = 0$, we get:
\eqntag{\label{211126145436}
	\FF_0 (x, f_0) = \sum_{m=0}^\infty \FF_{0m} (x) f_0^m = 0
\fullstop
}
By the Holomorphic Implicit Function Theorem, there is a domain $X_0 \subset X$ containing $x_0$ such that there is a unique holomorphic function $f_0 : X_0 \to \Complex$ that satisfies \eqntag{\label{211127154830}
	f_0 (x_0) = z_0 
	\qqtext{and}
	\FF_0 \big(x, f_0 (x) \big) = 0
	\qqquad
	\text{$\forall x \in X_0$\fullstop}
}
In fact, the domain $X_0$ can be chosen so small that
\eqntag{\label{211126150503}
	\JJ_0 = \JJ_0 (x) \coleq \evat{\frac{\del \FF_0}{\del z}}{\big(x, f_0 (x)\big)} = \sum_{m=0}^\infty m \FF_{0m} f_0^{m-1} \neq 0
	\qqquad
	\text{$\forall x \in X_0$\fullstop}
}

\textsc{Step 3: Subleading-order part.}
For clarity, let us examine equation \eqref{211204115845} at low orders in $\hbar$.
At order $n = 1$, it yields:
\eqns{
	0 =
	\sum_{m=0}^\infty
		\left(
			\sum_{i_1 + \ldots + i_m = 1}
				\FF_{0m} f_{i_1} \cdots f_{i_m}
			+
				\FF_{1m} f_0^m
		\right)
	&= \sum_{m=0}^\infty
		\Big(
			m \FF_{0m} f_0^{m-1} f_1
			+
				\FF_{1m} f_0^m
		\Big)
\\
	&= \left( \sum_{m=0}^\infty m \FF_{0m} f_0^{m-1} \right) f_1 + \sum_{m=0}^\infty \FF_{1m} f_0^m
\\
	&= \JJ_0 f_1 + \sum_{m=0}^\infty \FF_{1m} f_0^m
\fullstop
}

\vspace{-10pt}\enlargethispage{15pt}

Therefore, we are able to solve for $f_1$ uniquely:
\eqntag{\label{211126151910}
	f_1 \coleq -\JJ^{-1}_0 \sum_{m=0}^\infty \FF_{1m} f_0^m
\fullstop\vspace{-10pt}
}

Similarly, at order $n = 2$ in $\hbar$, we find:
\vspace{-10pt}
\eqns{
	0 
	&=
	\sum_{m=0}^\infty
		\left(
			\GREEN{\sum_{i_1 + \ldots + i_m = 2}
				\FF_{0m} f_{i_1} \cdots f_{i_m}}
			+ \BLUE{\sum_{i_1 + \ldots + i_m = 1}
				\FF_{1m} f_{i_1} \cdots f_{i_m}}
			+ \FF_{m2} f_0^m
		\right)
\\
	&=
	\sum_{m=0}^\infty
		\left(
			\GREEN{
				m \FF_{0m} f_0^{m-1} f_2 
				+ 
				\sum_{i_1 + \ldots + i_m = 2}
					\FF_{0m} f_{i_1} \cdots f_{i_m}
			}
			+ \BLUE{m \FF_{1m} f_0^{m-1} f_1}
			+ \FF_{m2} f_0^m
		\right)
\fullstop
}
Thus, once again, we are able to solve this equation for $f_2$ uniquely:
\eqntag{\label{211126152754}
	f_2 \coleq -\JJ^{-1}_0 \sum_{m=0}^\infty
		\left(
			\sum_{i_1 + \ldots + i_m = 2}
			\FF_{0m} f_{i_1} \cdots f_{i_m}
			+ m \FF_{1m} f_0^{m-1} f_1
			+ \FF_{m2} f_0^m
		\right)
\fullstop
}

\textsc{Step 4: Inductive step.}
More generally, at order $n \geq 1$ in $\hbar$, when we have already solved uniquely for functions $f_0, \ldots, f_{n-1}$, equation \eqref{211204115845} yields:
\eqns{\label{211126153030}
	0
	&=
	\sum_{m=0}^\infty \sum_{k=0}^n 
		\sum_{i_1 + \ldots + i_m = n-k}
		\FF_{km} f_{i_1} \cdots f_{i_m}
\\	&=
	\left(\sum_{m=0}^\infty m \FF_{0m} f_0^{m-1} \right) f_n
	+ \sum_{m=0}^\infty \sum_{k=0}^n 
		~\sum_{i_1 + \ldots + i_m = n-k}
		\FF_{km} f_{i_1} \cdots f_{i_m}
\fullstop{,}
}
which has a unique solution $f_n$ given by formula
\eqntag{\label{211127154827}
	f_n \coleq -\JJ^{-1}_0 \sum_{m=0}^\infty \sum_{k=0}^n 
		~\sum_{\substack{0 \leq i_1, \ldots, i_m \leq n - 1 \\ i_1 + \ldots + i_m = n-k}}
		\FF_{km} f_{i_1} \cdots f_{i_m}
\fullstop
\qedhere
}
\end{proof}

\subsection{Transformation to the Standard Form}

Next, we make a convenient change of variables in order to bring the given equation $\FF (x, \hbar, z) = 0$ to a standard form that is more easily handled using the Borel-Laplace method.
This transformation and the standard form are fully determined by the leading-order solution $f_0$ of the equation $\FF_0 (x,z) = 0$ and can always be achieved under our hypotheses.
Namely, we have the following statement.

\begin{lem}{211126191222}
Suppose $\FF$ is a holomorphic function on $X_0 \times S \times \Complex_z$ satisfying the hypotheses of \autoref{211126190846}.
Let $f_0$ and $f_1$ be the leading- and the next-to-leading-order parts of the formal solution $\hat{f}$ defined on $X_0 \subset X$.
Then the change of the unknown variable $z \mapsto w$ given by
\vspace{-5pt}
\eqntag{\label{211127161619}
	z = f_0 + \hbar (f_1 + w)\vspace{-5pt}
}
transforms the equation $\FF (x, \hbar, z) = 0$ into an equation in $w$ of the form
\vspace{-5pt}
\eqntag{\label{211127161740}
	w = \hbar \GG (x, \hbar, w)
\fullstop{,}\vspace{-5pt}
}
where $\GG$ is a holomorphic function uniquely determined by $f_0$ and $\FF$.
Furthermore, if $\FF$ admits a Gevrey asymptotic expansion as $\hbar \to 0$ along $\bar{A}$ uniformly for all $x \in X$ and locally uniformly for all $z \in \Complex_z$ and the domain $X_0$ is chosen so small that $\JJ_0$ (where $\JJ_0$ is the invertible holomorphic function on $X_0$ given by \eqref{211126150503}) is bounded from below on $X_0$, then $\GG$ also admits a Gevrey asymptotic expansion as $\hbar \to 0$ along $\bar{A}$ uniformly for all $x \in X_0$ and locally uniformly for all $z \in \Complex_z$.
Specifically, $\GG$ is defined by
\eqntag{\label{211209143638}
	\GG (x, \hbar, w) \coleq \hbar^{-1} \Big( w - \hbar^{-1}\JJ^{-1}_0 (x) \FF \big(x, \hbar, f_0 (x) + \hbar f_1 (x) + \hbar w\big)\Big)
\fullstop\vspace{-5pt}
}
\end{lem}

\begin{proof}
The only thing to check is that the righthand side of \eqref{211209143638} has no negative powers in $\hbar$.
In particular, since $\FF$ is an entire function in the variable $z$, identity \eqref{211209143638} makes it obvious that $\GG$ admits a uniform Gevrey asymptotic expansion $\hat{\GG}$ as $\hbar \to 0$ along $\bar{A}$ whenever $\JJ^{-1}_0$ is bounded and $\FF$ admits uniform Gevrey asymptotics.
A more explicit formula for the function $\GG$ is derived in \autoref{211209144324}.

\enlargethispage{10pt}
Let us now verify that $\GG$ has no negative powers in $\hbar$.
Clearly, the leading-order part of $\FF \big(\hbar, f_0 + \hbar f_1 + \hbar w\big)$ is simply $\FF_0 \big(x, f_0 (x) \big)$ which is zero because $f_0$ is the leading-order solution.
Therefore, the righthand side of \eqref{211209143638} is at worst of order $\hbar^{-1}$.
We argue that the next-to-leading-order part of $\FF \big(\hbar, f_0 + \hbar f_1 + \hbar w\big)$ is equal to $\JJ_0 w$.
We expand it as follows:
\vspace{-5pt}
\eqns{
	\Big[ \FF \big(\hbar, f_0 + \hbar f_1 + \hbar w\big) \Big]^{\OO (\hbar)}
	&= 
		\FF_1 (f_0) + \Big[ \FF_0 \big(f_0 + \hbar f_1 + \hbar w\big) \Big]^{\OO (\hbar)}
\\	&= 
		\sum_{m=0}^\infty \FF_{1m} f_0^m 
		+ \left[ \sum_{m=0}^\infty \FF_{0m} \big(f_0 + \hbar (f_1 + w)\big)^m \right]^{\OO (\hbar)}
\\
	&=
		\sum_{m=0}^\infty \FF_{1m} f_0^m 
		+ 	\left[ 
				\sum_{m=0}^\infty \sum_{i + j = m} \!\!\!\!
				\tbinom{m}{i, j}
				\FF_{0m}
				\big( f_0 \big)^{i} 
				\big(f_1 + w\big)^{j} \hbar^{j}
			\right]^{\OO (\hbar)}
\\
	&=
		\sum_{m=0}^\infty \FF_{1m} f_0^m 
		+
			\sum_{m=0}^\infty \sum_{i + j = m} \!\!\!\!
			m \FF_{0m} 
			\big( f_0 \big)^{m-1} 
			\big(f_1 + w\big)
\\
	&=
		\sum_{m=0}^\infty \FF_{1m} f_0^m 
		+
			\JJ_0 \big(f_1 + w\big)
\fullstop	
}
Identity \eqref{211126151910} shows that this expression equals $\JJ_0 w$, as desired.
\end{proof}

The analogue of the Formal Implicit Function Theorem (\autoref{211126190846}) for equations of the form \eqref{211127161740} is especially easy to formulate.

\begin{lem}{211203175809}
Let
\vspace{-5pt}
\eqntag{\label{211208181314}
	\hat{\GG} 
		= \hat{\GG} (x, \hbar, w) 
		\coleq \sum_{k=0}^\infty \GG_k (x, w) \hbar^k
\vspace{-5pt}
}
be any formal power series in $\hbar$ with holomorphic function coefficients on $X_0 \times \Complex_w$ for some domain $X_0 \subset \Complex^d_x$.
Then there is a unique formal power series
\vspace{-5pt}
\eqntag{\label{211203180204}
	\hat{g} = \hat{g} (x,\hbar) = \sum_{n=0}^\infty g_n (x) \hbar^n
\vspace{-5pt}
}
with holomorphic coefficients $g_k : X_0 \to \Complex$, which satisfies $\hat{g} (x, \hbar) = \hat{\GG} \big(x, \hbar, \hat{g} (x, \hbar)\big)$ for all $x \in X_0$.
In other words, the equation $w = \hbar \hat{\GG} (x, \hbar, w)$ has a unique formal power series solution $w = \hat{g}$.

In particular, if $S \subset \Complex_\hbar$ is a sectorial domain at the origin and $\GG$ is a holomorphic function $X_0 \times S \times \Complex_w \to \Complex$ which admits the power series $\hat{\GG}$ as a locally uniform asymptotic expansion as $\hbar \to 0$ in $S$, then the equation $w = \GG (x,\hbar,w) = 0$ has a unique formal power series solution $w = \hat{g} (x, \hbar)$ as above.

Moreover, $g_0 \equiv 0$ and all the higher-order coefficients $g_n$ are given by the following recursive formula:
\vspace{-5pt}
\eqntag{\label{211203180200}
	g_{n+1}
	= \sum_{k = 0}^{n} \sum_{m=0}^{n- k}
		\sum_{i_1 + \cdots + i_m = n-k} \!\!\!\!\!\!\!\!
		\GG_{km} g_{i_1} \cdots g_{i_m}
\fullstop{,}
}
where $\GG_{km} = \GG_{km} (x)$ are the coefficients of the double power series expansion
\eqntag{\label{211208181847}
	\hat{\GG} (x, \hbar, w)
		= \sum_{k=0}^\infty \sum_{m=0}^\infty \GG_{km} (x) \hbar^k w^m
\fullstop
}
\end{lem}

\begin{proof}
The proof is a computation very similar to the one in the proof of \autoref{211126190846}.
Plugging the solution ansatz \eqref{211203180204} into the double power series expansion \eqref{211208181847} of $\hat{\GG}$, the righthand side of the equation $w = \hbar \hat{\GG} (x, \hbar, w)$ becomes:
\eqns{
	\hbar \sum_{k=0}^\infty \sum_{m=0}^\infty \GG_{km} \hbar^{k} \left( \sum_{n=0}^\infty g_n \hbar^n \right)^m
		&= \hbar \sum_{n=0}^\infty \sum_{k=0}^\infty \sum_{m=0}^\infty \sum_{i_1 + \cdots + i_m = n}
			\GG_{km} g_{i_1} \cdots g_{i_m} \hbar^{k+n}
\\		
		&= \hbar \sum_{n=0}^\infty \sum_{k=0}^{n} \sum_{m=0}^{n-k} \sum_{i_1 + \cdots + i_m = n - k}
			\GG_{km} g_{i_1} \cdots g_{i_m} \hbar^{n}
\tag*{\qedhere}
\fullstop
}
\end{proof}

\paragraph{Explicit formula for the transformation.}
\label{211209144324}
We can derive a more explicit formula for the holomorphic function $\GG$ from \autoref{211126191222}.
First, let us write $\FF$ as a uniformly convergent power series in $z$:
\eqntag{\label{211127164058}
	\FF (x, \hbar, z) = \sum_{m=0}^\infty \AA_m (x, \hbar) z^m
\fullstop
}
For every $m \in \Natural$, let $\BB_m = \BB_m (x, \hbar)$ be a holomorphic function on $X \times S$ defined by the identity
\eqntag{\label{211127164021}
	\AA_m = \FF_{0m} + \hbar \FF_{1m} + \hbar^2 \BB_m
\fullstop
}
Then we substitute $z = f_0 + \hbar (f_1 + w)$ into $\FF (x, \hbar, z) = 0$ and expand using \eqref{211127164058} and \eqref{211127164021}:
\vspace{-10pt}
\eqns{
	0
	&= \FF (\hbar, f_0 + \hbar f_1 + \hbar w)
\\
	&= \GREEN{\sum_{m=0}^\infty \sum_{i+j+k = m} \!\!\!\! \tbinom{m}{i,j,k} \FF_{0m} f_0^i f_1^j w^k \hbar^{m-i}}
			+ \BLUE{\hbar \sum_{m=0}^\infty \sum_{i+j+k = m} \!\!\!\! \tbinom{m}{i,j,k} \FF_{1m} f_0^i f_1^j w^k \hbar^{m-i}}
\\	&\hspace{6cm}
			+ \hbar^2 \sum_{m=0}^\infty \sum_{i+j+k = m} \!\!\!\! \tbinom{m}{i,j,k} \BB_m f_0^i f_1^j w^k \hbar^{m-i}
\fullstop\vspace{-5pt}
}
Let us split the \GREEN{green} and \BLUE{blue} sums according to $\GREEN{[(i,j,k) = (m,0,0)] + [(i,j,k) = }$ $\GREEN{(m-1,1,0)] + [(i,j,k) = (m-1,0,1)] + [\text{the rest}]}$ and $\BLUE{[(i,j,k) = (m,0,0)] + [\text{the rest}]}	$, respectively.
We leave the black sum alone.
Thus, noting that
\eqn{
	\binom{m}{m,0,0} = 1
\qtext{and}
	\binom{m}{m-1,1,0} = \binom{m}{m-1,0,1} = m
\fullstop{,}
}
we get:
\newpage
\eqns{
	0&= \GREEN{\sum_{m=0}^\infty \FF_{0m} f_0^m 
		+ \left( \sum_{m=0}^\infty m \FF_{0m} f_0^{m-1} \right) f_1 \hbar
		+ \left( \sum_{m=0}^\infty m \FF_{0m} f_0^{m-1} \right) w \hbar}
\\
	&\quad \GREEN{+ \sum_{m=0}^\infty \sum_{\substack{i+j+k = m \\ 0 \leq i \leq m-2}} \!\!\!\! \tbinom{m}{i,j,k} \FF_{0m} f_0^i f_1^j w^k \hbar^{m-i}}
	+ \BLUE{\hbar \sum_{m=0}^\infty \FF_{1m} f_0^m}
\\
	&\quad + \BLUE{\hbar \sum_{m=0}^\infty \sum_{\substack{i+j+k = m \\ 0 \leq i \leq m-1}} \!\!\!\! \tbinom{m}{i,j,k} \FF_{1m} f_0^i f_1^j w^k \hbar^{m-i}}
	+ \hbar^2 \sum_{m=0}^\infty \sum_{i+j+k = m} \!\!\!\! \tbinom{m}{i,j,k} \BB_m f_0^i f_1^j w^k \hbar^{m-i}
\fullstop\vspace{-5pt}
\intertext{Now, the first \GREEN{green} sum vanishes because $f_0$ is a leading-order solution, see \eqref{211126145436}.
Likewise, the second \GREEN{green} sum and the first \BLUE{blue} sum add up to zero because $f_1$ is the next-to-leading-order solution, see \eqref{211126151910}.
In the third \GREEN{green} sum, the factor in front of $w\hbar$ is $-\JJ_0^{-1}$.
Let us also factorise $\hbar^2$ and another $\hbar$ out of the fourth \GREEN{green} and second \BLUE{blue} sums.
Altogether, we obtain:\vspace{-5pt}}
	0&= \GREEN{-\JJ_0^{-1} w\hbar 
		+ \hbar^2 \sum_{m=0}^\infty \sum_{\substack{i+j+k = m \\ 0 \leq i \leq m-2}} \!\!\!\! \tbinom{m}{i,j,k} \FF_{0m} f_0^i f_1^j w^k \hbar^{m-2-i}}
\\
	&\quad + \BLUE{\hbar^2 \sum_{m=0}^\infty \sum_{\substack{i+j+k = m \\ 0 \leq i \leq m-1}} \!\!\!\! \tbinom{m}{i,j,k} \FF_{1m} f_0^i f_1^j w^k \hbar^{m-1-i}}
	+ \hbar^2 \sum_{m=0}^\infty \sum_{i+j+k = m} \!\!\!\! \tbinom{m}{i,j,k} \BB_m f_0^i f_1^j w^k \hbar^{m-i}
\fullstop\vspace{-5pt}
\intertext{Next, it is convenient to rearrange the summations as follows:
\eqn{
	\sum_{m=0}^\infty \sum_{i+j+k = m} \!\!
		= \sum_{m=0}^\infty \sum_{k=0}^m \sum_{i+j = m - k} \!\!
		= \sum_{k=0}^\infty \sum_{m=k}^\infty \sum_{i+j = m - k}
\fullstop
}
Thus, we get:}
\\[-40pt]
	0&= \GREEN{-\JJ_0^{-1} w\hbar}
		+ \hbar^2 \sum_{k=0}^\infty \sum_{m=k}^\infty
			\left(
				\GREEN{\sum_{\substack{i+j = m - k \\ 0 \leq i \leq m-2}} \!\!\!\! \tbinom{m}{i,j,k} \FF_{0m} f_0^i f_1^j \hbar^{m-2-i}}
				+ \!\!\!\! \BLUE{\sum_{\substack{i+j = m - k \\ 0 \leq i \leq m-1}} \!\!\!\! \tbinom{m}{i,j,k} \FF_{1m} f_0^i f_1^j \hbar^{m-1-i}}
			\right.
\\
	&\hspace{9cm}
			\left.
				+ \!\!\!\! \sum_{i+j = m - k} \!\!\!\! \tbinom{m}{i,j,k} \BB_m f_0^i f_1^j \hbar^{m-i}
			\right)
			w^k
\fullstop
}
Multiplying through by $\JJ_0 \hbar^{-1}$ and taking $w$ over to the lefthand side, we obtain the equation $w = \hbar \GG (\hbar, w)$ where $\GG$ is defined as the power series
\vspace{-5pt}
\eqntag{\label{211127164333}
	\GG (x, \hbar, w) \coleq \sum_{k=0}^\infty \CC_k (x, \hbar) w^k
\fullstop{,}
\vspace{-5pt}
}
with coefficients $\CC_k$ given by the following formula:
\vspace{-10pt}
\begin{multline}
	\CC_k
		\coleq \JJ_0 \sum_{m=k}^\infty
			\left(
				\sum_{\substack{i+j = m - k \\ 0 \leq i \leq m-2}} \!\!\!\! \tbinom{m}{i,j,k} \FF_{0m} f_0^i f_1^j \hbar^{m-2-i}
				+ \!\!\!\! \sum_{\substack{i+j = m - k \\ 0 \leq i \leq m-1}} \!\!\!\! \tbinom{m}{i,j,k} \FF_{1m} f_0^i f_1^j \hbar^{m-1-i}
			\right.
\\
			\left.
				+ \!\!\!\! \sum_{i+j = m - k} \!\!\!\! \tbinom{m}{i,j,k} \BB_m f_0^i f_1^j \hbar^{m-i}
			\right)
\fullstop
\end{multline}

\subsection{Gevrey Regularity of the Formal Solution}

Now we show that the formal Borel transform of the formal solution $\hat{f}$ is a convergent power series in the Borel variable $\xi$; that is, the coefficients $f_n$ grow not faster than $n!$.
More precisely, we prove the following proposition.

\begin{prop}[{Gevrey Formal Implicit Function Theorem}]{211127150013}
\mbox{}\\
Assume all the hypotheses of \autoref{211126190846} and suppose in addition that the power series $\hat{\FF}$ is locally uniformly Gevrey on $X \times \Complex_z$.
Then $X_0 \subset X$ can be chosen so small that the formal power series $\hat{f}$ is uniformly Gevrey on $X_0$.
In particular, the formal Borel transform 
\eqntag{
	\hat{\phi} (x, \xi) =
	\hat{\Borel} [ \, \hat{f} \, ] (x, \xi)
		\coleq \sum_{n=0}^\infty \tfrac{1}{n!} f_{n+1} (x) \xi^n
}
is a uniformly convergent power series in $\xi$.
Concretely, if $X_0 \subset X$ is any subset where the function $\JJ_0$ is bounded from below and such that there are $\AA, \BB > 0$ such that $|\FF_k (x, z)| \leq \AA \BB^k k!$ for all $k \geq 0$, uniformly for all $x \in X_0$ and for all $z \in \Complex_z$ with $|z| < \RR$ for some $\RR > 0$, then there are constants $\CC, \MM > 0$ such that
\eqntag{
	\big| f_k (x) \big| \leq \CC \MM^k k!
\qqquad
	\text{$\forall x \in X_0, \forall k$\fullstop}
}
\end{prop}

\begin{proof}
Let $X_0 \subset X$ be such that the function $\JJ_0$ is bounded from below on $X_0$.
Then, by \autoref{211126191222}, the proof boils down to proving the following claim.

\textbf{Claim.}
\textit{
Assume all the hypotheses of \autoref{211203175809} and suppose that the power series $\hat{\GG}$ is Gevrey uniformly for all $x \in X_0$ and locally uniformly for all $w \in \Complex_w$.
Then the formal solution $\hat{g}$ is also uniformly Gevrey on $X_0$.
}

Let $\AA, \BB > 0$ be constants such that
\eqntag{\label{211126162722}
	\big| \GG_{km} (x) \big| \leq \AA \BB^{k+m} k!
\qqquad
	\text{$\forall x \in X_0, ~\forall k,m \in \Natural$\fullstop}
}
We will show that there is a constant $\MM > 0$ such that
\eqntag{\label{211126164059}
	\big| g_{n+1} (x) \big| \leq \MM^{n+1} n!
\qqquad
	\text{$\forall x \in X_0, ~\forall n \in \Natural$\fullstop}
}
This bound will be demonstrated in two main steps.
First, we will recursively construct a sequence $\set{\MM_n}_{n=0}^\infty$ of nonnegative real numbers such that
\eqntag{\label{211126164811}
	\big| g_{n+1} (x) \big| \leq \MM_{n+1} n!
\qqquad
	\text{$\forall x \in X_0, ~\forall n \in \Natural$\fullstop}
}
Then we will show that there is a constant $\MM > 0$ such that $\MM_n \leq \MM^n$ for all $n$.

\textsc{Step 1: Construction of $\set{\MM_n}_{n=0}^\infty$.}
Let $\MM_0 \coleq 0$.
We can take $\MM_1 \coleq \AA$ because $g_1 = \GG_{00}$.
Now we use induction on $n$ and formula \eqref{211203180200}, which is more convenient to rewrite as follows:
\eqntag{\label{211204105006}
	g_{n+1}
	= \sum_{m=0}^\infty \sum_{k=0}^n
		\sum_{i_1 + \cdots + i_m = n-k} \!\!\!\!\!\!\!\!
		\GG_{km} g_{i_1} \cdots g_{i_m}
\fullstop
}
Notice that, since $g_0 \equiv 0$, the sum over $i_1, \ldots, i_m$ is empty whenever $m > n - k$, so this expression really is the same as \eqref{211203180200}.
Assume that we have already constructed $\MM_0, \ldots, \MM_{n}$ such that $\big| g_{i} \big| \leq \MM_{i} (i-1)!$ for all $i = 0, \ldots, n$ and all $x \in X_0$.
Then we estimate $g_{n+1}$ using \eqref{211204105006}:
\eqns{
	|g_{n+1}|
		&\leq \sum_{m=0}^\infty \sum_{k=0}^n
		\sum_{i_1 + \cdots + i_m = n-k} \!\!\!\!\!\!\!\!
		\AA \BB^{k+m} k! \MM_{i_1} \cdots \MM_{i_m} (i_1 - 1)! \cdots (i_m - 1)!
\\
		&\leq \sum_{m=0}^\infty \sum_{k=0}^n
			\AA \BB^k  \!\!\!\!\!\!\!\!
		\sum_{i_1 + \cdots + i_m = n-k} \!\!\!\!\!\!\!\!
		\BB^m \MM_{i_1} \cdots \MM_{i_m} n!
}
Here, we used the fact that $i_1 + \cdots + i_m = n-k$ and the inequality $i!j! \leq (i+j)!$.
Thus, we can define
\eqntag{\label{211204111409}
	\MM_{n+1} \coleq \sum_{m=0}^\infty \sum_{k=0}^n
			\AA \BB^k  \!\!\!\!\!\!\!\!
		\sum_{i_1 + \cdots + i_m = n-k} \!\!\!\!\!\!\!\!
		\BB^m \MM_{i_1} \cdots \MM_{i_m}
\fullstop
}

\textsc{Step 2: Construction of $\MM$.}
To see that $\MM_n \leq \MM^n$ for some $\MM > 0$, we argue as follows.
Consider the following pair of power series in an abstract variable $t$:
\eqntag{
	\hat{p} (t) \coleq \sum_{n=0}^\infty \MM_n t^n
\qtext{and}
	\QQ (t) \coleq \sum_{m=0}^\infty \BB^m t^m
\fullstop
}
Notice that $\hat{p} (0) = \MM_0 = 0$ and that $\QQ (t)$ is convergent.
We will show that $\hat{p} (t)$ is also convergent.
The key is the observation that they satisfy the following equation:
\eqntag{\label{211204111930}
	\hat{p} (t) 
		= \AA t \QQ(t) \QQ\big( \hat{p}(t) \big)
		= \AA t \QQ(t) \sum_{m=0}^\infty \BB^m \hat{p}(t)^m
\fullstop
}
This equation was found by trial and error, and it is straightforward to verify directly by substituting the power series $\hat{p}(t)$ and $\QQ(t)$ and comparing the coefficients of $t^{n+1}$ using formula \eqref{211204111409}.

Now, consider the following holomorphic function in two complex variables $(t,p)$:
\eqn{
	\FF (t,p) \coleq - p + \AA t \QQ(t) \QQ(p)
}
It has the following properties:
\eqn{
	\FF (0,0) = 0
\qqtext{and}
	\evat{\frac{\del \FF}{\del p}}{(t,p) = (0,0)} = -1 \neq 0
\fullstop
}
By the Holomorphic Implicit Function Theorem, there exists a unique holomorphic function $p (t)$ near $t = 0$ such that $p (t) = 0$ and $\FF \big(t, p (t)\big) = 0$.
Thus, $\hat{p} (t)$ must be the convergent Taylor series expansion at $t = 0$ for $p(t)$, so its coefficients grow at most exponentially: there is a constant $\MM > 0$ such that $\MM_n \leq \MM^n$.
\end{proof}

\subsection{Exact Perturbation Theory}

Now we show that the convergent Borel transform $\hat{\phi} (x, \xi)$ of the formal solution admits an analytic continuation along a ray in the Borel $\xi$-plane and furthermore its Laplace transform is well-defined.
First, we prove the following lemma.

\begin{lem}{211123193427}
Let $X_0 \subset \Complex^d_x$ be a domain.
Let $S \coleq \set{\hbar ~\big|~ \Re (1/\hbar) > 1/\RR} \subset \Complex_\hbar$ be the Borel disc of some diameter $\RR > 0$.
Recall that the opening of $S$ is $A_+ \coleq (-\pi/2, +\pi/2)$.
Let $\GG = \GG (x, \hbar, w)$ be a holomorphic function on $X_0 \times S \times \Complex_w$ which admits a Gevrey asymptotic expansion
\eqntag{\label{211215174141}
	\GG (x, \hbar, w) \simeq \hat{\GG} (x, \hbar, w)
\quad
\text{as $\hbar \to 0$ along $\bar{A}_+$\fullstop{,}}
}
uniformly for all $x \in X_0$ and locally uniformly for all $w \in \Complex_w$.
Then there is a Borel disc $S_0 \coleq \set{\hbar ~\big|~ \Re (1/\hbar) > 1/\RR_0} \subset S$ of possibly smaller diameter $\RR_0 \in (0, \RR]$ such that there is a unique holomorphic function $g = g(x, \hbar)$ on $X_0 \times S_0$ which admits a uniform Gevrey asymptotic expansion
\eqntag{\label{211206171855}
	g (x, \hbar) \simeq \hat{g} (x, \hbar)
\quad
\text{as $\hbar \to 0$ along $\bar{A}_+$, unif. $\forall x \in X_0$\fullstop{,}}
}
and such that $g(x, \hbar) = \hbar \GG \big(x, \hbar, g(x,\hbar) \big) = 0$ for all $(x,\hbar) \in X_0 \times S_0$.
Furthermore, $g$ is the uniform Borel resummation of $\hat{g}$: for all $(x,\hbar) \in X_0 \times S_0$,
\eqntag{\label{211210124905}
	g (x, \hbar) = \cal{S} \big[ \: \hat{g} \: \big] (x, \hbar)
\fullstop
}
\end{lem}

\begin{proof}
First, uniqueness of $g$ follows from the asymptotic property \eqref{211206171855}.
Indeed, suppose $g'$ is another such function.
Then the difference $g - g'$ is a holomorphic function on $X_0 \times S_0$ which is uniformly Gevrey asymptotic to $0$ as $\hbar \to 0$ along the closed arc $\bar{A}_+$ of opening angle $\pi$.
By Nevanlinna's Theorem (\cite[pp.44-45]{nevanlinna1918theorie} and \cite{MR558468}; see also \cite[Theorem B.11]{MY2008.06492}), there can only be one holomorphic function on $S_0$ (namely, the constant function $0$) which is Gevrey asymptotic to $0$ as $\hbar \to 0$ along $\bar{A}_+$.
Thus, $g - g'$ must be identically zero.

To construct $g$, start by expanding $\GG$ as a uniformly convergent power series in $w$:
\eqntag{\label{211204131509}
	\GG (x, \hbar, w)
		= \sum_{m=0}^\infty \AA_m (x, \hbar) w^m
\fullstop
}

\paragraph*{Step 1: The Borel Transform.}
Let $a_m = a_m (x)$ be the $\hbar$-leading-order part of $\AA_m$ and let $\alpha_m (x, \xi) \coleq \Borel \big[ \AA_m \big] (x, \xi)$.
By the assumption \eqref{211215174141}, there is some $\epsilon > 0$ such that $\alpha_m$ is a holomorphic function on $X_0 \times \Xi$, where
\eqn{
	\Xi \coleq \set{\xi ~\big|~ \op{dist} (\xi, \Real_+) < \epsilon}
\fullstop{,}
}
with uniformly at most exponential growth at infinity in $\xi$ (cf. \autoref{211215140949}), and
\eqntag{\label{211215185405}
	\AA_m (x, \hbar) = a_m (x) + \Laplace \big[\, \alpha_m \,\big] (x, \hbar)
}
for all $(x, \hbar) \in X_0 \times S$ provided that the diameter $\RR$ is sufficiently small.

Dividing through by $\hbar$ and applying the analytic Borel transform to the equation $w = \hbar \GG (x, \hbar, w)$, we obtain the following nonlinear ordinary differential equation:
\eqntag{\label{211123195506}
	\del_\xi \sigma = \alpha_0 + \sum_{m=1}^\infty \Big( a_m \sigma^{\ast m} + \alpha_m \ast \sigma^{\ast m} \Big)
\fullstop{,}
}
where the unknown variables $w$ and $\sigma$ are related by $\sigma = \Borel [w]$ and $w = \Laplace [\sigma]$.
A solution of \eqref{211123195506} with initial condition $\sigma (x, 0) = a_0 (x)$ is equivalently the solution of the integral equation
\eqntag{\label{211123195801}
	\sigma = a_0 + \int_0^\xi \left[ \alpha_0 + \sum_{m=1}^\infty \Big( a_m \sigma^{\ast m} + \alpha_m \ast \sigma^{\ast m} \Big) \right] \dd{t}
\fullstop
}

\paragraph*{Step 2: Method of Successive Approximations.}
We solve this integral equation using the method of successive approximations.
To this end, define a sequence of holomorphic functions $\set{\sigma_n}_{n=0}^\infty$ on $X_0 \times \Xi$ as follows:
\eqntag{\label{211123200337}
	\sigma_0 \coleq a_0,
\qquad
	\sigma_1 \coleq \int\nolimits_0^\xi \big[ \alpha_0 + a_1 \sigma_0 \big] \dd{t}
\fullstop{,}
}
and for $n \geq 2$ by
\eqntag{\label{211123200340}
	\sigma_n \coleq \int\nolimits_0^\xi 
		\sum_{m=1}^n 
			\left( a_m  \hspace{-10pt} \sum_{i_1 + \cdots + i_m = n - m} \hspace{-15pt}
				\sigma_{i_1} \ast \cdots \ast \sigma_{i_m}
				+ \alpha_m \ast \hspace{-20pt} \sum_{i_1 + \cdots + i_m = n - m - 1} 
				\hspace{-15pt} \sigma_{i_1} \ast \cdots \ast \sigma_{i_m}
			\right) \dd{t}
\fullstop
}

\textbf{Main Technical Claim.}
\textit{The infinite series 
\eqntag{\label{211124091608}
	\sigma (x, \xi) \coleq \sum_{n=0}^\infty \sigma_n (x, \xi)
}
converges uniformly for all $(x, \xi) \in X_0 \times \Xi$ and defines a holomorphic solution of the integral equation \eqref{211123195801} with uniformly at-most-exponential growth at infinity in $\xi$; that is, there are constants $\DD, \KK > 0$ such that
\eqntag{\label{211204172313}
	\big| \sigma (x, \xi) \big| \leq \DD e^{\KK |\xi|}
\qquad
\text{$\forall (x, \xi) \in X_0 \times \Xi$\fullstop}
}
Furthermore, the formal Borel transform
\eqntag{
	\hat{\sigma} (x, \xi) =
	\hat{\Borel} [ \, \hat{g} \, ] (x, \xi)
		= \sum_{n=0}^\infty \tfrac{1}{n!} g_{n+1} (x) \xi^n
}
of the formal solution $\hat{g} (x, \hbar)$ is the Taylor series expansion of $\sigma$ at $\xi = 0$.
}

\enlargethispage{10pt}

The assertion of \autoref{211123193427} follows from this claim by defining
\eqntag{\label{211215174705}
	g (x, \hbar) \coleq \Laplace [\sigma] (x, \hbar) = \int_0^{+\infty} e^{-\xi/\hbar} \sigma (x, \xi) \dd{\xi}
\fullstop
}
Indeed, the exponential estimate \eqref{211204172313} implies that the Laplace transform of $\sigma$ is uniformly convergent for all $(x, \hbar) \in X_0 \times S_0$ where $S_0 = \set{\hbar ~\big|~ \Re (1/\hbar) > 1/\RR_0}$ as long as $\RR_0 < \KK^{-1}$.
We now turn to the proof of the Main Technical Claim.

\paragraph*{Step 3: Solution Check.}
First, assuming that the infinite series $\sigma$ is uniformly convergent for all $(x, \xi) \in X_0 \times \Xi$, we check that it satisfies \eqref{211123195801} by direct substitution.
Start by removing $\sigma_0$ and $\sigma_1$ via the following manipulation:
\eqns{
&\phantom{=}\;\;
	a_0 + \int_0^\xi 
		\left[ \BLUE{\alpha_0
			+ \sum_{m=1}^\infty a_m \left(\:\sum_{n=0}^\infty \sigma_n\right)^{\!\!\!\ast m}} 
			+ \sum_{m=1}^\infty \alpha_m \ast \left(\:\sum_{n=0}^\infty \sigma_n\right)^{\!\!\!\ast m}
		\right] \dd{t}
\\
&= a_0 + \int_0^\xi 
		\left[ \BLUE{\alpha_0 + a_1 \sigma_0 + a_1 \sum_{n=1}^\infty \sigma_n
			+ \sum_{m=2}^\infty a_m \left(\:\sum_{n=0}^\infty \sigma_n\right)^{\!\!\!\ast m}}
			+ \sum_{m=1}^\infty \alpha_m \ast \left(\:\sum_{n=0}^\infty \sigma_n\right)^{\!\!\!\ast m} 
		\right] \dd{t}
\\
&= \sigma_0 + \BLUE{\sigma_1} + \int_0^\xi
		\left[ \BLUE{a_1 \sum_{n=1}^\infty \sigma_n
			+ \sum_{m=2}^\infty a_m \left(\:\sum_{n=0}^\infty \sigma_n\right)^{\!\!\!\ast m}}
			+ \sum_{m=1}^\infty \alpha_m \ast \left(\:\sum_{n=0}^\infty \sigma_n\right)^{\!\!\!\ast m} 
		\right] \dd{t}
\fullstop
}
The goal is to show that the the integral in the above expression is $\sum_{n \geq 2} \sigma_n$.
Focus now on the expression inside the integral and use the formula
\eqntag{\label{211124093929}
	\left(\:\sum_{n=0}^\infty \sigma_n\right)^{\!\!\!\ast m}
		= \sum_{n=0}^\infty \sum_{i_1 + \cdots + i_m = n} \hspace{-15pt}
				\sigma_{i_1} \ast \cdots \ast \sigma_{i_m}
\fullstop
}
Then we manipulate it as follows:
\eqns{
&\phantom{=}\;\;
	a_1 \GREEN{\sum_{n=1}^\infty \sigma_n}
	+ \sum_{m=2}^\infty a_m \GREEN{\left(\:\sum_{n=0}^\infty \sigma_n\right)^{\!\!\!\ast m}}
	+ \sum_{m=1}^\infty \alpha_m \ast \GREEN{\left(\:\sum_{n=0}^\infty \sigma_n\right)^{\!\!\!\ast m}}
\\
&= a_1 \GREEN{\sum_{n=1}^\infty \sigma_n}
	+ \sum_{m=2}^\infty a_m
		\GREEN{\sum_{n=0}^\infty \sum_{i_1 + \cdots + i_m = n}
		\hspace{-15pt}
		\sigma_{i_1} \ast \cdots \ast \sigma_{i_m}}
	+ \sum_{m=1}^\infty \alpha_m \ast
		\GREEN{\sum_{n=0}^\infty \sum_{i_1 + \cdots + i_m = n}
		\hspace{-15pt}
		\sigma_{i_1} \ast \cdots \ast \sigma_{i_m}}
\fullstop
\intertext{Now we shift the summation index $n$ up by $1$ in the first green sum, by $m$ in the second, and by $m+1$ in the third:}
&= a_1 \GREEN{\sum_{n=\ORANGE{2}}^\infty \sigma_{\ORANGE{n-1}}}
	+ \sum_{m=2}^\infty a_m
		\GREEN{\sum_{n=\ORANGE{m}}^\infty \hspace{-10pt} \sum_{ \substack{ i_1, \ldots, i_m \geq 0 \\ i_1 + \cdots + i_m = \ORANGE{n - m}}}
		\hspace{-20pt}
		\sigma_{i_1} \ast \cdots \ast \sigma_{i_m}}
	+ \sum_{m=1}^\infty \alpha_m \ast \hspace{-5pt} 
		\GREEN{\sum_{n=\ORANGE{m+1}}^\infty \hspace{-10pt} \sum_{ \substack{ i_1, \ldots, i_m \geq 0 \\ i_1 + \cdots + i_m = \ORANGE{n - m -1}}} 
		\hspace{-20pt}
		\sigma_{i_1} \ast \cdots \ast \sigma_{i_m}}
\fullstop
\intertext{Notice that all terms in the second \GREEN{green} sum with $n < m$ are zero, so we can start the summation over $n$ from $n = 2$ (which is the lowest possible value of $m$) without altering the result.
Similarly, all terms in the third \GREEN{green} sum with $n < m + 1$ are zero, so we can start from $n = 2$.
The first \GREEN{green} sum is left unaltered.
Thus:}
&= a_1 \GREEN{\sum_{n=2}^\infty \sigma_{n-1}}
	+ \sum_{m=2}^\infty a_m
		\GREEN{\sum_{n=\ORANGE{2}}^\infty \hspace{-10pt} \sum_{ \substack{ i_1, \ldots, i_m \geq 0 \\ i_1 + \cdots + i_m = n - m}}
		\hspace{-20pt}
		\sigma_{i_1} \ast \cdots \ast \sigma_{i_m}}
	+ \sum_{m=1}^\infty \alpha_m \ast
		\GREEN{\sum_{n=\ORANGE{2}}^\infty \hspace{-10pt} \sum_{ \substack{ i_1, \ldots, i_m \geq 0 \\ i_1 + \cdots + i_m = n - m -1}} 
		\hspace{-20pt}
		\sigma_{i_1} \ast \cdots \ast \sigma_{i_m}}
\fullstop
\intertext{The advantage of this way of expressing the sums is that we can now interchange the summations over $m$ and $n$ in the second and third green sums:}
&= a_1 \GREEN{\sum_{n=2}^\infty \sigma_{n-1}}
	+ \GREEN{\sum_{n=2}^\infty}
		\sum_{m=2}^\infty a_m \hspace{-15pt}
		\GREEN{\sum_{ \substack{ i_1, \ldots, i_m \geq 0 \\ i_1 + \cdots + i_m = n - m}}
		\hspace{-20pt}
		\sigma_{i_1} \ast \cdots \ast \sigma_{i_m}}
	+ \GREEN{\sum_{n=2}^\infty}
		\sum_{m=1}^\infty \alpha_m \ast \hspace{-25pt}
		\GREEN{\sum_{\substack{ i_1, \ldots, i_m \geq 0 \\ i_1 + \cdots + i_m = n - m -1}} 
		\hspace{-20pt}
		\sigma_{i_1} \ast \cdots \ast \sigma_{i_m}}
\displaybreak
\\
&= \GREEN{\sum_{n=2}^\infty}
	\left\{
		a_1 \GREEN{\sigma_{n-1}}
		+ \sum_{m=2}^\infty a_m \hspace{-15pt} 
			\GREEN{\sum_{ \substack{ i_1, \ldots, i_m \geq 0 \\ i_1 + \cdots + i_m = n - m}}
			\hspace{-20pt}
			\sigma_{i_1} \ast \cdots \ast \sigma_{i_m}}
		+ \sum_{m=1}^\infty \alpha_m \ast \hspace{-25pt}
			\GREEN{\sum_{\substack{ i_1, \ldots, i_m \geq 0 \\ i_1 + \cdots + i_m = n - m -1}} 
			\hspace{-20pt}
			\sigma_{i_1} \ast \cdots \ast \sigma_{i_m}}
	\right\}
\fullstop
\intertext{Notice that the term $a_1 \sigma_{n-1}$ fits well into the first sum over $m$ to give the $m=1$ addend.
Thus:}
&= \GREEN{\sum_{n=2}^\infty}
	\left\{
		\sum_{m=1}^\infty a_m \hspace{-15pt} 
			\GREEN{\sum_{ \substack{ i_1, \ldots, i_m \geq 0 \\ i_1 + \cdots + i_m = n - m}}
			\hspace{-20pt}
			\sigma_{i_1} \ast \cdots \ast \sigma_{i_m}}
		+ \sum_{m=1}^\infty \alpha_m \ast \hspace{-25pt}
			\GREEN{\sum_{\substack{ i_1, \ldots, i_m \geq 0 \\ i_1 + \cdots + i_m = n - m -1}} 
			\hspace{-20pt}
			\sigma_{i_1} \ast \cdots \ast \sigma_{i_m}}
	\right\}
\\
&= \GREEN{\sum_{n=2}^\infty} \sum_{m=1}^\infty
	\left\{
		a_m \hspace{-15pt} 
			\GREEN{\sum_{ \substack{ i_1, \ldots, i_m \geq 0 \\ i_1 + \cdots + i_m = n - m}}
			\hspace{-20pt}
			\sigma_{i_1} \ast \cdots \ast \sigma_{i_m}}
		+  \alpha_m \ast \hspace{-25pt}
			\GREEN{\sum_{\substack{ i_1, \ldots, i_m \geq 0 \\ i_1 + \cdots + i_m = n - m -1}} 
			\hspace{-20pt}
			\sigma_{i_1} \ast \cdots \ast \sigma_{i_m}}
	\right\}
\fullstop
\intertext{Finally, notice that that both sums are empty for $m > n$, so:}
&= \GREEN{\sum_{n=2}^\infty} \sum_{m=1}^{\ORANGE{n}}
	\left\{
		a_m \hspace{-15pt} 
			\GREEN{\sum_{ \substack{ i_1, \ldots, i_m \geq 0 \\ i_1 + \cdots + i_m = n - m}}
			\hspace{-20pt}
			\sigma_{i_1} \ast \cdots \ast \sigma_{i_m}}
		+  \alpha_m \ast \hspace{-25pt}
			\GREEN{\sum_{\substack{ i_1, \ldots, i_m \geq 0 \\ i_1 + \cdots + i_m = n - m -1}} 
			\hspace{-20pt}
			\sigma_{i_1} \ast \cdots \ast \sigma_{i_m}}
	\right\}
\fullstop
}
The sum over $m$ is precisely the expression inside the integral in \eqref{211123200340} defining $\sigma_n$.
This shows that $\sigma$ satisfies the integral equation \eqref{211123195801}.

\paragraph*{Step 4: Convergence.}
Now we show that $\sigma$ is a uniformly convergent infinite series on $X_0 \times \Xi$ and therefore defines a holomorphic function.
In the process, we also establish the estimate \eqref{211204172313}.

Let $\BB, \CC, \LL > 0$ be such that for all $(x, \xi) \in X_0 \times \Xi$ and all $m \in \Natural$,
\eqntag{\label{211204171819}
	\big| a_m (x) \big| \leq \CC \BB^m
\qqtext{and}
	\big| \alpha_m (x, \xi) \big| \leq \CC \BB^m e^{\LL |\xi|}
\fullstop
}
We claim that there are $\DD, \MM > 0$ such that for all $(x, \xi) \in X_0 \times \Xi$ and $n \in \Natural$,
\eqntag{\label{211204172105}
	\big| \sigma_n (x, \xi) \big| \leq \DD \MM^n \frac{|\xi|^n}{n!} e^{\LL |\xi|}
\fullstop
}
If we achieve \eqref{211204172105}, then the uniform convergence and the exponential estimate \eqref{211204172313} both follow at once because
\eqn{
	\big| \sigma (x, \xi) \big|
		\leq \sum_{n=0}^\infty \big| \sigma_n (x, \xi) \big|
		\leq \sum_{n=0}^\infty \DD \MM^n \frac{|\xi|^n}{n!} e^{\LL |\xi|}
		\leq \DD e^{(\MM + \LL) |\xi|}
\fullstop
}
To demonstrate \eqref{211204172105}, we proceed in two steps.
First, we construct a sequence of positive real numbers $\set{\MM_n}_{n=0}^\infty$ such that for all $n$ and all $(x, \xi) \in X_0 \times \Xi$,
\eqntag{\label{211204173527}
	\big| \sigma_n (x, \xi) \big| \leq \MM_n \frac{|\xi|^n}{n!} e^{\LL |\xi|}
\fullstop
}
We will then show that there are constants $\DD, \MM$ such that $\MM_n \leq \DD \MM^n$ for all $n$.

\paragraph*{Step 4.1: Construction of $\set{\MM_n}$.}
We can take $\MM_0 \coleq \CC$ and $\MM_1 \coleq \CC (1 + \BB \MM_0)$ because $\sigma_0 = a_0$ and
\eqn{
	\big| \sigma_1 \big|
	\leq \int_0^\xi \big( |\alpha_0| + |a_1| |\sigma_0| \big) |\dd{t}|
	\leq \CC (1 + \BB \MM_0) \int_0^{|\xi|} e^{\LL r} \dd{r}
	\leq \CC (1 + \BB \MM_0) |\xi| e^{\LL |\xi|}
\fullstop{,}
}
where in the final step we used \autoref{180824194855}.
Let us assume now that we have constructed $\MM_0, \ldots, \MM_{n-1}$ such that $|\sigma_i| \leq \MM_i \frac{|\xi|^i}{i!} e^{\LL |\xi|}$ for all $i = 0, \ldots, n-1$.
Then we use formula \eqref{211123200340} together with \autoref{180824194855} and \autoref{211205075846} in order to derive an estimate for $\sigma_n$:
\eqns{
	|\sigma_n|
	&\leq \int_0^\xi \sum_{m=1}^n
			\left( |a_m|  \hspace{-15pt} \sum_{ i_1 + \cdots + i_m = n - m} \hspace{-15pt}
				\big| \sigma_{i_1} \ast \cdots \ast \sigma_{i_m} \big|
				+ \hspace{-20pt} \sum_{ i_1 + \cdots + i_m = n - m - 1} 
				\hspace{-15pt} \big| \alpha_m \ast \sigma_{i_1} \ast \cdots \ast \sigma_{i_m} \big|
			\right) |\dd{t}|
\\
	&\leq \sum_{m=1}^n
			\left( \CC \BB^m \hspace{-20pt} \sum_{ i_1 + \cdots + i_m = n - m } \hspace{-15pt}
				\MM_{i_1} \cdots \MM_{i_m}
				+ \CC \BB^m \hspace{-20pt} \sum_{ i_1 + \cdots + i_m = n - m - 1 } 
				\hspace{-15pt} \MM_{i_1} \cdots \MM_{i_m}
			\right) 
			\int_0^\xi \frac{|t|^{n-1}}{(n-1)!} e^{\LL |t|} |\dd{t}|
\\
	&\leq \sum_{m=1}^n
			\CC \BB^m
			\left( ~\sum_{ i_1 + \cdots + i_m = n - m } \hspace{-15pt}
				\MM_{i_1} \cdots \MM_{i_m}
				+ \hspace{-10pt} \sum_{ i_1 + \cdots + i_m = n - m - 1 } 
				\hspace{-15pt} \MM_{i_1} \cdots \MM_{i_m}
			\right)
			\frac{|\xi|^n}{n!} e^{\LL |\xi|}
\fullstop
}
Thus, this expression allows us to define the constant $\MM_n$ for $n \geq 2$.
In fact, a quick glance at this formula reveals that it can be extended to $n = 0, 1$ by defining
\eqntag{\label{211206124551}
	\MM_n \coleq
		\sum_{m=0}^n
			\CC \BB^m
			\left( ~\sum_{ i_1 + \cdots + i_m = n - m } \hspace{-15pt}
				\MM_{i_1} \cdots \MM_{i_m}
				+ \hspace{-10pt} \sum_{ i_1 + \cdots + i_m = n - m - 1 } 
				\hspace{-15pt} \MM_{i_1} \cdots \MM_{i_m}
			\right)
\fullstop
}
Indeed, if $m = 0$, then the two sums inside the brackets are nonzero only when $n = 0$, so we recover $\MM_0 = \CC$.
Likewise, if $n = 1$, then the only the terms $m = 0$ and $m = 1$ are nonzero in this formula, and they are respectively $0 + \CC$ and $\CC \BB \MM_0 + 0$, so we recover $\MM_1$ defined previously.

\paragraph*{Step 4.2: Bounding $\MM_n$.}
To see that $\MM_n \leq \DD \MM^n$ for some $\DD, \MM > 0$, consider the following power series in an abstract variable $t$:
\eqntag{\label{211206124848}
	\hat{p} (t) \coleq \sum_{n=0}^\infty \MM_n t^n
\qqtext{and}
	\QQ (t) \coleq \sum_{m=0}^\infty \CC \BB^m t^m
\fullstop
}
Notice that $\QQ (t)$ is a convergent and $\QQ (0) = \CC = \MM_0$.
We will show that $\hat{p} (t)$ is also a convergent power series.
The key observation is that $\hat{p}$ satisfies the following functional equation:
\eqntag{\label{211206125154}
	\hat{p} (t) = (1+t) \QQ \big( t \hat{p} (t) \big)
\fullstop
}
This equation was found by trial and error, and it is straightforward to verify by direct substitution of the expansions \eqref{211206124848} and comparing the coefficients of $t^n$ using the defining formula \eqref{211206124551} for $\MM_n$.
Explicitly, the righthand side of \eqref{211206125154} expands as follows:
\eqns{
	&\phantom{=}~~
	(1+t) \sum_{m=0}^\infty \CC \BB^m \left( t \sum_{n=0}^\infty \MM_n t^n \right)^m
\displaybreak
\\
	&= (1+t) \sum_{n=0}^\infty \sum_{m=0}^\infty \CC \BB^m
		\hspace{-10pt} \sum_{ i_1 + \cdots + i_m = n } \hspace{-15pt}
		\MM_{i_1} \cdots \MM_{i_m} t^{n+m}
\\
	&= (1+t) \sum_{n=0}^\infty \sum_{m=0}^\infty \CC \BB^m
		\hspace{-10pt} \sum_{ i_1 + \cdots + i_m = \ORANGE{n-m} } \hspace{-15pt}
		\MM_{i_1} \cdots \MM_{i_m} t^{\ORANGE{n}}
\\
	&= \sum_{n=0}^\infty \sum_{m=0}^\infty \CC \BB^m
		\left( ~\sum_{ i_1 + \cdots + i_m = n - m } \hspace{-15pt}
				\MM_{i_1} \cdots \MM_{i_m} t^n
				+ \hspace{-10pt} \sum_{ i_1 + \cdots + i_m = n - m } 
				\hspace{-15pt} \MM_{i_1} \cdots \MM_{i_m} t^{n+1}
		\right)
\\
	&= \sum_{n=0}^\infty \sum_{m=0}^{\ORANGE{n}} \CC \BB^m
			\left( ~\sum_{ i_1 + \cdots + i_m = n - m } \hspace{-15pt}
				\MM_{i_1} \cdots \MM_{i_m}
				+ \hspace{-10pt} \sum_{ i_1 + \cdots + i_m = \ORANGE{n - m - 1} } 
				\hspace{-15pt} \MM_{i_1} \cdots \MM_{i_m}
			\right) t^n
\fullstop
}
Now, consider the following holomorphic function in two variables $(t,p)$:
\eqntag{
	\FF (t, p) \coleq - p + (1+t) \QQ (tp)
\fullstop
}
It has the following properties:
\eqn{
	\FF (0, \CC) = 0
\qqtext{and}
	\evat{\frac{\del \PP}{\del p}}{(t,p) = (0, \CC)} = - 1 \neq 0
\fullstop
}
By the Holomorphic Implicit Function Theorem, there exists a unique holomorphic function $p(t)$ near $t = 0$ such that $p(0) = \CC$ and $\FF \big(t, p(t)\big) = 0$.
Therefore, $\hat{p} (t)$ must be the convergent Taylor series expansion of $p(t)$ at $t = 0$, so its coefficients grow at most exponentially: there are constants $\DD, \MM > 0$ such that $\MM_n \leq \DD \MM^n$.
This completes the proof of the Main Technical Claim and hence of \autoref{211123193427}.
\end{proof}

At last, we are able to collect all our work in order to finish the proof of the Gevrey Asymptotic Implicit Function Theorem (\autoref{211125100306}) with $\NN = 1$.

\begin{proof}[Proof of \autoref{211125100306} ($\NN = 1$)]
By the Formal Implicit Function Theorem (\autoref{211126190846}), there is a subdomain $X_0 \subset X$ containing $x_0$ such that the equation $\FF (x, \hbar, z) = 0$ has a unique formal solution $\hat{f}$ satisfying $f_0 (x_0) = z_0$.
Let $f_0, f_1$ be its leading- and next-to-leading-order parts in $\hbar$.
As in \autoref{211126191222}, we change variables as $z = f_0 + \hbar f_1 + \hbar w$ to transform the equation $\FF (x, \hbar, z) = 0$ into $w = \hbar \GG (x, \hbar, w)$.
By \autoref{211123193427}, this equation has a unique holomorphic solution $w = g(x,\hbar)$ on $X_0 \times S_0$ for some sectorial subdomain $S_0 \subset S$ still with opening $A$ and admitting a uniform Gevrey asymptotic expansion as $\hbar \to 0$ along $\bar{A}$.
Finally, we define $f \coleq f_0 + \hbar f_1 + \hbar g$ which is readily seen to have all the desired properties.
\end{proof}

\end{appendices}

\newpage
\begin{adjustwidth}{-2cm}{-1.5cm}
{\footnotesize
\bibliographystyle{nikolaev}
\bibliography{/Users/Nikita/Documents/Library/References}
}
\end{adjustwidth}
\end{document}